\documentclass[reqno]{amsart}
\usepackage{color}
\usepackage{amsmath,amssymb,amsfonts,amsthm,bm}
\usepackage{mathrsfs}
\usepackage{graphicx}
\usepackage{enumerate}
\usepackage[numbers, sort&compress]{natbib}
\usepackage[shortlabels]{enumitem}
  \usepackage{times}
   \usepackage[varvw]{newtxmath}
  
 \usepackage[reversemp,paperwidth=155mm,paperheight=235mm,top={22mm},headheight={5.5pt},headsep={5.6mm},text={121mm,195mm},marginparsep=5mm,marginparwidth=12mm, bindingoffset=6mm, footskip=10mm]{geometry}

\numberwithin{equation}{section}
 
 \usepackage[colorlinks,
             linkcolor=black,
           anchorcolor=black,
            citecolor=black
           ]{hyperref}

\newtheorem{theorem}{Theorem}[section]
\newtheorem{corollary}[theorem]{Corollary}
\newtheorem{lemma}[theorem]{Lemma}
\newtheorem{proposition}[theorem]{Proposition}
\theoremstyle{assumption}

\theoremstyle{definition}
\newtheorem{definition}[theorem]{Definition}
\theoremstyle{remark}
\newtheorem{remark}[theorem]{Remark}
\numberwithin{equation}{section}

\newcommand{\eps}{\varepsilon}
\newcommand{\norm}[1]{\Vert#1\Vert}

\newcommand{\inner}[1]{\left(#1\right)}
\newcommand{\comi}[1]{\left<#1\right>}

\newcommand{\normm}[1]{{ \vert\kern-0.25ex \vert\kern-0.25ex \vert #1
		\vert\kern-0.25ex \vert\kern-0.25ex \vert}}

\makeatletter

 \newbox \abstractbox
\renewenvironment{abstract}{\global\setbox\abstractbox=\vbox\bgroup
 \hsize=\textwidth
  \vskip 1.2cm
  \noindent\unskip \textbf{Abstract.}
 }
 {
 \egroup}

\@namedef{subjclassname@2020}{%
	\textup{2020} Mathematics Subject Classification}

\def\@settitle{%
  \bgroup
  \centering
  \vglue1cm
  \fontsize{12}{15}\fontseries{b}\selectfont
  \@title
  \vskip 20pt plus 6pt minus 8pt
  \egroup
}

\def\@setauthors{%
  \begingroup
  \trivlist
  \centering \bfseries
 \normalsize\@topsep30\p@\relax
  \advance\@topsep by -\baselineskip
  \item\relax
  \andify\authors
 {\rmfamily\authors}%
  \endtrivlist
  \endgroup
}

\def\@setaddresses{\par
  \nobreak \begingroup
\normalsize
  \def\author##1{\nobreak\addvspace\bigskipamount}%
  \def\\{\unskip, \ignorespaces}%
  \interlinepenalty\@M
  \def\address##1##2{\begingroup
    \par\addvspace\bigskipamount\noindent
    \@ifnotempty{##1}{(\ignorespaces##1\unskip) }%
    {\ignorespaces##2}\par\endgroup}%
  \def\curraddr##1##2{\begingroup
    \@ifnotempty{##2}{\nobreak\indent{\itshape Current address}%
      \@ifnotempty{##1}{, \ignorespaces##1\unskip}\/:\space
      ##2\par}\endgroup}%
  \def\email##1##2{\begingroup
    \@ifnotempty{##2}{\nobreak\noindent{\itshape E-mail address}%
      \@ifnotempty{##1}{, \ignorespaces##1\unskip}\/:
       ##2\par}\endgroup}%
   \def\urladdr##1##2{\begingroup
    \@ifnotempty{##2}{\nobreak\indent{\itshape URL}%
      \@ifnotempty{##1}{, \ignorespaces##1\unskip}\/:\space
      \ttfamily##2\par}\endgroup}%
  \addresses
  \endgroup
}

  \renewcommand\section{\@startsection{section}{1}{\z@}%
 {27pt plus 6pt minus 8pt}{14pt plus 6pt minus 8pt}
 {\center\normalfont\large\bfseries}}

 \def\subsection{\@startsection{subsection}{2}%
  \z@{.5\linespacing\@plus.7\linespacing}{.2\linespacing}%
  {\normalfont\bfseries}}
\def\subsubsection{\@startsection{subsubsection}{3}%
  \z@{.5\linespacing\@plus.7\linespacing}{-.5em}%
  {\normalfont\bfseries}}

\makeatother

\begin{document}
	
		\title[Global Sobolev Well-Posedness of the MHD Boundary
Layer Equations]{Global well-posedness of the MHD boundary layer equations in the  Sobolev Space}

	\author[W.-X. Li, Z. Xu and A. Yang]{Wei-Xi Li, Zhan Xu and  Anita  Yang}
	
	\address[W.-X. Li]{School of Mathematics and Statistics, \&  Hubei Key Laboratory of Computational Science,  Wuhan University 
	Wuhan 430072, China} 
	\email{wei-xi.li@whu.edu.cn}

\address[Z. Xu]{School of Mathematics and Statistics, Wuhan University,  
	Wuhan 430072,    China} \email{xuzhan@whu.edu.cn}
	
\address[A. Yang]{Department of Mathematics, The Chinese University of Hong Kong, Shatin, Hong Kong} \email{ayang@math.cuhk.edu.hk}


	\begin{abstract}
 We study the two-dimensional MHD boundary layer equations.  For small perturbation around a tangential background magnetic field, we obtain   the global-in-time existence and uniqueness  of solutions in Sobolev spaces.  The proof relies on the novel  combination of  the well-explored  cancellation mechanism and the idea of linearly-good unknowns,  and   we use the former idea to deal with the top tangential derivatives and the latter one admitting fast decay rate to control lower-order  derivatives.    
\end{abstract}

	\maketitle
	
	 \tableofcontents
	
	\section{Introduction and main result}
Magnetohydrodynamic (MHD) is the study of the interaction between  electromagnetic fields and conducting fluids,  describing  the motion of conducting fluid under the influence of the  magnetic fields.  In this work we consider the   two-dimensional  (2D) case  and assume the fluid moves 
	in the half-space $\Omega$. 
Then  in the incompressible framework,  the governing equations read
\begin{equation}\label{mhdsys}
\left\{
	\begin{aligned}
		&\partial_t \bm  U +(\bm  U \cdot\nabla) \bm U-(\bm B \cdot\nabla) \bm B +\nabla   P -\frac{1}{\mathrm {Re}} \Delta \bm U  =0,\\
		&\partial_t \bm  B  -\nabla\times(\bm U\times\bm B) -\frac{1}{\mathrm{Rm}} \Delta \bm B=0,\\
		&\nabla\cdot \bm U=\nabla\cdot \bm B=0,\\
		&{\bm U}|_{t=0}={\bm U_{in}},\quad {\bm B}|_{t=0}={\bm B_{in}},
	\end{aligned}
\right.
\end{equation}
where ${\bm U}$ and $\bm B$ are the velocity and magnetic fields, respectively, and   Re and Rm stand for the hydrodynamic and magnetic Reynolds numbers, respectively.    Moreover, suppose  
system \eqref{mhdsys} is complemented   with  an uniform background magnetic field and  the no-slip boundary condition
    on the velocity  field,  saying 
 \begin{eqnarray*}
 	\bm U|_{\partial\Omega}= {\boldmath 0}, \quad \bm B |_{\partial\Omega}=\bm e,
 \end{eqnarray*}
for some constant vector $\bm e$. Formally,  as the Reynolds numbers   Re and Rm go to $+\infty$  we expect  system \eqref{mhdsys} will trend to the ideal  MHD system. However, in the case with physical boundaries,  the  mathematically rigorous justification of this process of high Reynolds limit    is  more complicated and interesting, due to the mismatched   boundary conditions between \eqref{mhdsys} and the limiting ideal system which lead  to the appearance of a boundary layer.  In the specific case when Re $\approx$ Rm, following the argument of Prandtl's ansatz,    G\'{e}rard-Varet  and Prestipino \cite{MR3657241} derived the following governing equations for the boundary layer:  
	\begin{equation}\label{eq1}
		\left\{
		\begin{aligned}
			&  \partial_t u+(u\partial_x+v\partial_y)u+  \partial_x p -\partial_y^2u
			=(b_1\partial_x  +b_2\partial_y)b_1,\\
			&  \partial_t b_1+(u\partial_x+v\partial_y)b_1-\partial_y^2b_1
			=(b_1\partial_x +b_2\partial_y)u, \\
			& \partial_xu+\partial_yv=0,\quad \partial_xb_1+\partial_yb_2=0, 
		\end{aligned}
		\right.
	\end{equation}
 which are posed on  the half-space $\mathbb R^2_{+}=\{(x,y)\in \mathbb R^2, y >0\}$, and complemented with the boundary condition 
 \begin{equation}\label{bc}
 	(u,v)|_{y=0}=(0,0)\  \textrm{ and } \ (b_1,b_2)|_{y=0}={\bm e}, \  \  \lim\limits_{y\to +\infty}(u,b_1)=\big (\alpha(t,x), \beta(t,x)\big ),
 \end{equation}
 and the initial condition 
 \begin{eqnarray}\label{ic}
 	(u,b_1)|_{t=0}=(u_{in},b_{in}).
 \end{eqnarray}
Here $(u,v)$ and $(b_1,b_2)$  are the velocity   and  the magnetic fields, respectively, and $(\alpha(t,x),\beta(t,x),p(t,x))$ are given functions of $(t,x),$  standing for the outflow of velocity, magnetic and pressure, which satisfy the Bernoulli's law:
 \begin{equation*}
		\left\{
		\begin{aligned}
&(\partial_t +\alpha\cdot \partial_x)\alpha+\partial_x p-\beta\partial_x\beta=0,\\
&(\partial_t +\alpha\cdot \partial_x)\beta-\beta\partial_x\alpha=0.
		\end{aligned}
		\right.
\end{equation*}
Without loss of generality,  we will consider a uniform outflow $(\alpha,\beta)=(0,1).$  Moreover,  we consider $(b_1,b_2)$ as a small perturbation around the uniform  background magnetic field $\bm e=(1,0)$. Then   
  using the notation that  
\begin{align*}
	(b_1, b_2)=(1+f,  g), 
\end{align*}
we may rewrite the initial-boundary problem  \eqref{eq1}-\eqref{ic}   as
	\begin{equation}
		\left\{
		\begin{aligned}
			&  \partial_t u+(u\partial_x+v\partial_y)u  -\partial_y^2u
			=(1+f) \partial_x f+g\partial_y  f,\\
			&  \partial_t   f+(u\partial_x+v\partial_y)  f-\partial_y^2  f
			=(1+f)\partial_x u+g\partial_yu, \\
			& \partial_xu+\partial_yv=0,\quad \partial_x  f+\partial_yg=0,\\
   			& (u,v,f,g)|_{y=0}=(0,0,0,0),\quad \lim\limits_{y\to +\infty}(u,  f)=(0,0),\\
      			& (u,  f)|_{t=0}=(u_{in},  f_{in}),\\
		\end{aligned}
		\right.
		\label{eq2}
	\end{equation}
where $ f_{in}= b_{in}-1$.
 This paper aims to investigate the global-in-time existence and uniqueness of solutions to the MHD layer system \eqref{eq2} in the Sobolev space.
 
 {\bf Notations.} Before  stating the main  result, we first list some notations to be used throughout the paper.  We will use  $\norm{\cdot}_{L^2}$ and $\inner{\cdot, \cdot}_{L^2}$ to denote the norm and inner product of  $L^2=L^2(\mathbb R_+^2)$   and use the notation   $\norm{\cdot}_{L_y^2}$ and $\inner{\cdot, \cdot}_{L_y^2}$  when the variable $y$ is specified. Similar notation  will be used for $L^\infty$.  Moreover, denote by  $L^p_{x}L^q_y=L^p(\mathbb R; L^q(\mathbb{R}_+))$  the classical Lebesgue space, and  similarly for the Sobolev space $H^p_{x}H^q_y.$ For  any given norm $\norm{\cdot}$ we use the convention that 
 \begin{equation*}
 	\norm{A}:=\Big(\sum_{1\leq j\leq k} \norm{A_j}^2\Big)^{1\over2} 
 \end{equation*}
 for vector-valued functions $A=(A_1,A_2,\cdots, A_k)$.

 We define  the weighted Lebesgue space with a given non-negative function $\omega$  by setting 
 \begin{equation*}
 	L_\omega^2:=\Big\{h: \mathbb{R}_{+}^2\to\mathbb{R}; \  \norm{h}_{L_\omega^2}:=\Big(\int_{\mathbb R_+^2} \omega  h^2dxdy\Big)^{1\over2}<+\infty\Big\}, 
 \end{equation*}
 which is equipped with the norm  $\norm{\cdot}_{L^2_\omega}$ and the inner product $\inner{\cdot, \cdot}_{L^2_\omega}$. 
More generally,  define 
\begin{equation}\label{hmkomega}
H^{m,k}_{\omega}:=\Big\{h: \mathbb{R}_{+}^2\to\mathbb{R}; \  \  \norm{h}_{H^{m,k}_{\omega}}:=\Big(\sum^m_{i=0}\sum^k_{j=0}\norm{\partial_x^{i}\partial_y^{j}h}^2_{L^2_{\omega}}\Big)^{1\over2}<+\infty\Big\}.
\end{equation}
In this paper we will work with  the following time-dependent weight function:
\begin{equation}\label{defmu}
 \mu:=\mu(t,y)=\exp\Big(\frac{y^2}{4\comi t}\Big),
\end{equation}
where here and below
\begin{equation*}
	\comi t:=1+t. 
\end{equation*}
More generally, we define  
\begin{equation}\label{mulambda}
	\mu_{\lambda}:=\mu_{\lambda}(t,y)=\exp\Big(\frac{\lambda y^2}{4\comi t}\Big),\quad 0\leq  \lambda\leq 1. 
\end{equation}
The parameter $\lambda$ in \eqref{mulambda} is allowed to be negative so that we have 
\begin{equation}\label{integr}
\forall\  \lambda<0,  \ \forall\ t\geq 0, \quad 	\norm{\mu_{\lambda} }_{L_y^2}\leq C_\lambda  \comi t^{1\over 4}
\end{equation}
with constant $C_\lambda$   depending only on $\lambda.$

We will solve  system  \eqref{eq2} in the weighted Sobolev space $H^{8,0}_\mu\cap H^{5,1}_{\mu_\lambda}$ with any $0\leq \lambda<1$.  So the initial data $u_{in}, f_{in}$ in \eqref{eq2} satisfy the compatibility condition that
\begin{equation}\label{comcon}
	 (u_{in},f_{in})|_{y=0}=\lim\limits_{y\to +\infty}	 (u_{in},f_{in}) =(0,0).
\end{equation}

	\begin{theorem}\label{th1}
 Assume   the initial data  in \eqref{eq2} satisfy that   $ u_{in},f_{in} \in  H^{8,0}_{\mu_{in}}\cap H^{5,1}_{\mu_{in}}$ with
 \begin{equation*}
 	\mu_{in}:=\mu(0,y)=\exp\Big (\frac{y^2}{4}\Big).
 \end{equation*}
 Moreover suppose that the compatibility condition \eqref{comcon}  is fulfilled and  that 
\begin{align}\label{eq4}
  \forall\ x\in\mathbb R,\quad   \int^{+\infty}_0 u_{in}(x, y)dy=0 \ \mathrm{and} \ \int^{+\infty}_0   f_{in}(x,y)dy=0.
\end{align}
Let $0< \delta\leq \frac{1}{25}$ be an arbitrarily given number. Then for any $0\leq \lambda<1$, there exists a small constant $\eps_0>0$ and  a constant   $C>0,$ both depending on $\lambda$ and $\delta$,  such that  if
\begin{equation}\label{smallassumption}
    \norm{(u_{in},  f_{in})}_{H_{\mu_{in}}^{8,0}}+ \norm{(u_{in},  f_{in})}_{H_{\mu_{in}}^{5,1}}\leq \varepsilon_0,
\end{equation}
then the MHD boundary layer system \eqref{eq2} admits a unique global-in-time solution $ u, f \in L^\infty\big([0,+\infty[\ ; \  H^{8,0}_{\mu} \cap H^{5,1}_{\mu_{\lambda}}\big)$,  satisfying that
\begin{align*} 
  &\sup_{t\geq 0}    \comi t^{\frac{1-\delta}{4}}\norm{(u, f)}_{H^{8,0}_{\mu}}+   \Big ( \int_0^{+\infty}  \comi t^{\frac{1-\delta}{2}}\norm{(\partial_yu, \partial_yf)}_{H^{8,0}_{\mu}}^2 dt\Big)^{1\over2} \leq C\varepsilon_0,
\end{align*}
and that
\begin{multline*}
 \sup_{t\geq 0}  \sum_{0\leq k\leq 1} \comi t^{\frac{5-\delta}{4}+\frac{k}{2}}\norm{(\partial_y^ku, \partial_y^kf)}_{H^{5,0}_{\mu_{\lambda}}} \\
  +\bigg( \sum_{0\leq k\leq 1}  \int_0^{+\infty} \comi t^{\frac{5-\delta}{2}+k}\norm{(\partial_y^{k+1}u, \partial_y^{k+1}f)}_{H^{5,0}_{\mu_{\lambda}}}^2\, dt\bigg)^{1\over2} \leq C\varepsilon_0.   
\end{multline*}
Recall $\mu,\mu_\lambda$ are defined in \eqref{defmu} and \eqref{mulambda}, respectively.  
 \end{theorem}

\begin{remark}
	Without the smallness assumption \eqref{smallassumption}, the MHD boundary layer system \eqref{eq2} is usually ill-posed in the Sobolev space, seeing \cite[Theorem 1.2]{MR4102162}. 
\end{remark}

\begin{remark}
	As to be seen below,  the proof of Theorem \ref{th1} combines the cancellation mechanism in \cite{MR4342301}  and the idea of linearly-good unknowns in \cite{MR4271962}. Precisely,  we will apply the cancellation mechanism for the  treatment of  top tangential derivatives and meanwhile the linearly-good unknowns  idea when dealing with lower-order derivatives. The former technique is used to overcome the loss of derivatives and the advantage of  the latter one is that  the Sobolev norm of the  new good unknowns decays faster than  that of  the velocity and magnetic fields themselves.  
 More details can be found in Subsection \ref{subsec:method} below. 
 \end{remark}

 We briefly review  some related  works on the  Prandtl and MHD boundary layer equations.  
In the absence of magnetic fields,
the MHD layer   system \eqref{eq2}  is  reduced to  the  classical Prandtl equation which was derived by L. Prandtl in 1904. 
The mathematical study on   the Prandtl boundary layer has a  long history, and   the well/ill-posedness property   is  well-explored in various function spaces,  seeing e.g. \cite{MR3327535, MR3795028,  MR3925144,  MR1476316, MR2601044, MR2849481,  MR3461362, MR3284569, MR2975371,  MR3590519, MR3493958,  MR2020656, MR4465902, MR3710703,MR3464051,MR4293727} and the references therein.  To explore the well-posedness theory of the Prandtl and MHD boundary layer equations,  the main difficulty lies in the loss of derivatives which occurs in the nonlocal terms
\begin{equation*}
	v(t,x,y)=-\int_0^y \partial_x u(t,x,\tilde y) d\tilde y\ \textrm{ and }\  g(t,x,y)=-\int_0^y \partial_x f(t,x,\tilde y) d\tilde y.
\end{equation*}
Compared with the Prandtl equations,  the MHD layer system \eqref{eq2} has additional  degeneracy property caused by the normal component $g$ of magnetic fields.    In order to overcome the loss of derivatives,   there are basically two main settings for the well-posedness theory of the Prandtl and  MHD boundary layer equations,  based on  whether or not  the structural assumption is imposed.   One refers to the analytic or more general Gevrey spaces  without any structural assumption on the initial data,  and  another one  refers to function spaces
with finite order of regularity under some  structural conditions  in two space dimensions. 

For the 2D Prandtl equation,  Oleinik's monotonicity condition plays a crucial role when investigating the well-posedness in spaces with finite order of regularity,  seeing  e.g. Oleinik-Samokhin's book \cite{MR1697762} for the local existence and uniqueness of classical solutions  by using   the Crocco transformation.  Using a new idea of
   cancellation mechanism,  two   groups  Alexandre-Wang-Xu-Yang \cite{MR3327535} and  Masmoudi-Wong 
\cite{MR3385340}  independently proved the local Sobolev well-posedness 
  through the energy method.  Hence, the loss of  tangential  derivatives can be overcame by using either Crocco transformation or cancellation mechanism under the monotonicity condition. For  the MHD boundary layer equations, thanks to the stabilizing effect  of  magnetic fields,     the Sobolev well-posedness in  the two-dimensional  case  is also available  even when the velocity field does not satisfy Oleinik's monotonicity assumption;  In fact,      inspired by  the idea of  cancellation   mechanism  initiated by \cite{MR3327535,MR3385340},    Liu-Xie-Yang \cite{MR3882222} established the local Sobolev well-posedness    under the hypothesis that the tangential  component of magnetic  fields does not degenerate. 
  
   Without any structural assumption on initial data, motived by the abstract Cauchy-Kowalewski theorem  it is natural to work with  the analytic functions space to overcome  the loss of one-order derivative.  The local well-posedness theory in the analytic functions setting was proved by Sammartino-Caflisch \cite{MR1617542}. Furthermore, by virtue of the intrinsic cancellation structure of the Prandtl equations, 
   the analyticity regularity may be  relaxed to general  Gevrey class. In fact,  based on a new understanding of the cancellation mechanism,  Dietert and G\'{e}rard-Varet \cite{MR3925144} and   Li-Masmoudi-Yang \cite{MR4465902} proved  in the 2D and 3D cases, respectively,   that the Prandtl equations are indeed local well-posed  in Gevrey 2 space.   
   Note that  the Gevrey index $2$ in \cite{MR3925144,MR4465902}  is optimal  in view of the ill-posedness theory  in the work \cite{MR2601044} of  G\'{e}rard-Varet and Dormy.  Hence   the optimality of  the well/ill-posedness theory  for the 2D and 3D Prandtl equations without any structural  assumption is well justified.  However, for the counterpart of the MHD boundary layer system, the Gevrey well-posedness is more complicated because of  the additional loss of derivatives in   magnetic fields.  For this, so far only Gevrey index up to $\frac32$  rather than $2$ is achieved  for  the local well-posedness for the 2D and 3D MHD boundary systems(cf. \cite{MR4270479}).  Compared with the Prandtl counterparts \cite{MR3925144,MR4465902}, this Gevrey index $\frac 32$ is expected not to be critical  for the well-posedness theory of the MHD boundary layer equations.

The aforementioned   works are mainly about the local existence.   Much less is known for the global well-posedness theory.   Here we  mention  Xin-Zhang's work \cite{MR2020656} on the global existence of weak solutions to the 2D Prandtl equations  under an additional favorable pressure condition.  On the other hand,  without the monotonicity condition, boundary layer separation is well expected (see e.g. \cite{MR4028516,MR4289038} for the justification of the Goldstein singularity)  and there are  lots of works on the instability phenomena.  Interested readers may refer to the works of E-Engquist \cite{MR1476316} about the construction of
blowup solutions for large  data (see  also Collot-Ghoul-Masmoudi \cite{MR4264948} for the formulation of singularity),  Grenier
\cite{MR1761409} on the unstable Euler shear flow that yields instability of the  Prandtl equations, and   \cite{MR2601044, MR2952715, MR3464020, MR2849481} about the instability around
  shear flows.   
     
    However,  for small and analytic data  we may    expect the global  solutions to the Prandtl and MHD boundary layer equations even 
   when the structural assumptions  are violated.  In fact,   
inspired by the  nonlinearly-good unknown   idea  in Masmoudi-Wong' work \cite{MR3385340},  Ignatova-Vicol \cite{MR3461362} introduced  a  new linearly-good  unknown with  decay rate faster  than that of the velocity itself and then take advantage of the better decay rate to construct  the almost global analytic solutions for the 2D Prandtl equations, after the earlier work \cite{MR3464051} of Zhang-Zhang  on the long time behavior.    This idea of linearly-good unknown  is  further developed by Paicu-Zhang \cite{MR4271962},  where the global analytic solution  is established  by  introducing another new linearly-good unknown with  the analytic norm  decaying  fast enough.   We refer to Wang-Wang-Zhang \cite{MR4701733} for the further generalization of   \cite{MR4271962} to the Gevrey setting,  and  \cite{MR4293727,MR4213671} for the  application of the linearly-good unknown idea to the global  well-posedness of the MHD  
 boundary layer equations in the analytic setting.  Note the works mentioned above are limited to the analytic or Gevrey setting and the global existence in Sobolev spaces remains open. The purpose of this work is to study the global well-posedness of the MHD boundary layer system \eqref{eq2} in the Sobolev setting.

The organization of the paper is as follow. In Section \ref{sec:apri} we state the key {\it a priori} estimate for proving Theorem \ref{th1}  and     
discuss the idea in the proof of the {\it a priori} estimate. Sections \ref{SecP}-\ref{SecE2} are devoted to proving the {\it a priori} estimate. In Appendices  \ref{appendix1} and \ref{appumfm}  we present the derivation of the evolution equations  obeyed by the auxiliary functions introduced in this text.  

\section{\emph{A priori} estimate}\label{sec:apri}
The key part for proving Theorem \ref{th1} is 
 to derive     {\it a priori} energy estimate for system \eqref{eq2}.   Then  the global-in-time existence and uniqueness   will follow  from  a standard  regularization procedure. 
  Hence, we only present in detail the proof of  the   {\it a priori} estimate stated in Subsection \ref{subsec:method}  and omit the  regularization procedure  for brevity.
  
\subsection{Cancellation mechanism and    good unknowns}

  Because of the loss of tangential derivatives  in   nonlocal terms, the Prandtl   and  MHD boundary layer equations  are usually ill-posed in the Sobolev setting (cf. \cite{MR2601044,MR4102162} for instance).  To overcome the degeneracy caused by the absence of the horizontal diffusion,     two groups Alexandre-Wang-Xu-Yang \cite{MR3327535} and  Masmoudi-Wong 
\cite{MR3385340} introduced independently the following nonlinearly-good unknown for the Prandtl equations: 
\begin{equation}\label{s}
	\Phi:= \partial_yu-\frac{\partial_y^2u}{\partial_yu}u=(\partial_yu)\partial_y\Big(\frac{u}{\partial_yu}\big),
	 \end{equation}
 provided Oleinik's monotonicity condition is fulfilled (i.e., $\partial_yu\neq 0$). Note  the quantity $\Phi$ above behaves better than $u$ and $\partial_yu$,  and in fact   $\Phi$   satisfies  an evolution equation where the terms involving the loss of tangential derivatives  are   cancelled  out.   This  advantage, with the help of Hardy's inequality, enables  to conclude the local-in-time well-posedness of the 2D Prandtl equation in Sobolev spaces  (cf. \cite{MR3327535, MR3385340} for more details). Inspired  by this cancellation idea,  under the hypothesis that  the tangential component of magnetic field does not degenerate,   Liu-Xie-Yang \cite{MR3882222} proved the local well-posedness for the 2D MHD boundary layer system, basing on   the 
 observation of  a new cancellation mechanism. This justifies the stabilizing effect of magnetic fields.    
 
 The idea in \cite{MR3385340}   was further developed by   Ignatova-Vicol \cite{MR3461362},  where the authors introduced  a new linearly-good unknown of the following form:
\begin{equation}\label{iv}
	\Psi:=\partial_yu-\frac{\partial_y^2\phi(t,y)}{\partial_y\phi(t,y)}  u=(\partial_y\phi)  \partial_y \Big(\frac{u}{\partial_y\phi}\Big), 
\end{equation}  
with $\phi(t,y)$ well chosen  function  such that 
$-\frac{\partial_y^2\phi(t,y)}{\partial_y\phi(t,y)}$ is the heat self-similar variable $\frac{y}{2\comi t}$ (e.g. $\partial_y\phi=\mu$  with $\mu $  defined in \eqref{defmu}).  Note that   Oleinik's monotonicity condition does not hold anymore in \cite{MR3461362},  so that in order to overcome the one-order loss of tangential derivative, the strong  tangential analyticity is required for initial data. 
By a new observation that  the analytic norm of the  good unknown $\Psi$ in \eqref{iv} decays faster than that of  $u$ itself,  Ignatova-Vicol \cite{MR3461362} established the almost global-in-time   solution to the 2D Prandtl equation  in the space of  analytic functions. Motivated by  $\Psi$ in \eqref{iv},    Paicu-Zhang \cite{MR4271962}  worked with another new  linearly-good unknown:
\begin{equation}\label{gu}
	G=u+\frac{y}{2\comi t}\varphi \  \textrm{ with }\  \partial_y\varphi= u. 
\end{equation}
 Compared with   the quantity $\Psi$  in \eqref{iv},  the analytical norm of $\partial_yG$ decays almost like $\comi t^{-{7\over4}}$,  faster than the decay rate   of   the analytical norm of $\Psi $ which is almost like $\comi t^{-{5\over4}}$. 
Basing on this  new  decay rate,    Paicu-Zhang \cite{MR4271962}  established the global-in-time existence  of solutions to the 2D Prandtl equation with small analytic data. 
 
 \subsection{Methodology and statement of \emph{a priori} estimate}\label{subsec:method}
This work is strongly inspired by the idea of cancellation mechanism and good unknowns   recalled  in the previous part.  On one side, the cancellation mechanism works well for the Sobolev well-posedness,  but only the local-in-time existence  may be expected. On the other side, one may get the global-in-time  well-posedness using   the idea of linearly-good unknowns, but  ask strong analytic or Gevrey regularity for initial data. Rather than the sole use of the cancellation mechanism or the linearly-good unknown idea in the aforementioned works,  we will combine the two techniques in this work.
 
  To investigate the Sobolev well-posedness of    system \eqref{eq2}, 
the main difficulty arises from
 the loss of   tangential derivatives in the terms $v$ and $g$.   To overcome this degeneracy, we adopt the cancellation mechanism observed in \cite{MR4342301}, which is inspired by the earlier work \cite{MR3882222}.  Precisely, define 
 \begin{equation}\label{eq2.1}
    u_m:=
    \left\{
    \begin{aligned}
    &	 \partial_x^mu+\frac{\partial_yu}{1+f}\partial_x^{m-1}g, \quad \textrm{if } \ m\geq 1,\\
    & u,\quad \textrm{if } \ m=0,
    \end{aligned}
    \right.
\end{equation}
and 
 \begin{equation}\label{eqfm}
   f_m:=
    \left\{
    \begin{aligned}
    &	 \partial_x^m f +\frac{\partial_y f }{1+f}\partial_x^{m-1}g, \quad \textrm{if } \ m\geq 1,\\
    & f,\quad \textrm{if } \ m=0.
    \end{aligned}
    \right.
\end{equation}
Similar to the nonlinearly-good unknown $\Phi$ in \eqref{s},  in the evolution equations  of  $u_m$ and $f_m$  above,  the terms involving the loss of derivatives are cancelled out, that is (see Appendix \ref{appumfm} for more detail),
\begin{equation}\label{equmfm}
	\left\{
    \begin{aligned}
&\inner{\partial_t+u\partial_x+v\partial_y-\partial_y^2}u_m-\big((1+f)\partial_x+g\partial_y\big)f_m =l.o.t.,\\
&\inner{\partial_t+u\partial_x+v\partial_y-\partial_y^2}f_m-\big((1+f)\partial_x+g\partial_y\big)u_m =l.o.t.,
\end{aligned}
\right.
\end{equation}
 where $l.o.t.$    refers to lower-order terms.   This enables to perform energy estimates in the Sobolev setting for the above system.  On the other hand, the decay rate of the Sobolev norms of $u_m$ and $f_m$ is not fast enough to ensure the global-in-time existence.  For this we introduce  the good unknown in \eqref{gu} and define two  linearly-good unknowns   $\mathcal{U}$ and $\mathcal{F}$ by setting   
\begin{equation}\label{eq3+}
    \mathcal{U}(t,x,y):=u(t,x,y)+\frac{y}{2\comi t}\int^{y}_0 u(t,x, \tilde y)d\tilde y,\quad     \mathcal{F}:= f +\frac{y}{2\comi t}\int^{y}_0 f d\tilde y. 
\end{equation} 
Applying $\partial_y$ to the equations above  yields that   \begin{align*}
 \int_0^y ud\tilde y =2\comi t \big(\partial_y\mathcal{U}-\partial_yu\big)-yu,\quad    \int_0^y fd\tilde y =2\comi t \big(\partial_y\mathcal{F}-\partial_yf\big)-yf. 
\end{align*}
Then if  the terms on the right-hand side trend to $0$ as $y$ going to $+\infty$, we have 
 \begin{equation}\label{zeromean}
 \int^{+\infty}_0 udy=0  \  \textrm{ and } \   \int^{+\infty}_0 fdy=0.
 \end{equation}
Thus we may  rewrite $\mathcal{U}$ and $\mathcal F$ in \eqref{eq3+} as 
 \begin{align}\label{eq3}
    \mathcal{U}=u-\frac{y}{2\comi t}\int^{+\infty}_yud\tilde y,\quad     \mathcal{F}= f -\frac{y}{2\comi t}\int^{+\infty}_y f d\tilde y. 
\end{align} 
Different from system \eqref{equmfm},    there is no cancellation mechanism in the evolution equations solved by  the top tangential derivatives  of  $\mathcal U$ and $\mathcal F$,  so that the loss of tangential derivatives may occur.  However,  we can perform Sobolev estimates for the lower order derivatives of  $\mathcal U$ and $\mathcal F.$  When  dealing with the quadratic nonlinear terms,    we will  control top tangential derivatives of $u$ and $ f$ in terms of $u_m$ and $f_m$ (see Lemma \ref{lemu8} below),  and   the lower order  derivatives of $u $ and $ f$ by the $H^{5,1}_\mu$-norms of  $\mathcal U$ and  $ \mathcal F$ (see Lemma \ref{lemu0} below).   Roughly speaking, $u_m$ and $f_m$ behave  as the heat equation so that the decay rate in the $L^2_\mu$ space is almost like $-\frac{1}{4}$ order with respect to time (see Lemma \ref{techlema} below). On the other hand,  $\mathcal U$ and $\mathcal F$ behave  as a heat equation with an additional damping term $\frac{1}{\comi t}$, which yields that  the decay rate of the $L^2_\mu$ of $\mathcal U$ and $\mathcal F$ is almost like $-\frac{5}{4}$ order  (see Remark \ref{rmkest} below).  This motives us to define the energy   as follows. 

\begin{definition}\label{defed}
	Let $\delta$ be an   arbitrarily fixed number, 
satisfying that
\begin{equation*}
	 0<\delta\leq  \frac{1}{25}.
\end{equation*}
 We   
define  the energy $\mathcal E_\delta=\mathcal E_\delta(t) $  by
\begin{equation}\label{edelta}
\mathcal E_\delta:=  \sum_{m\leq 8}\comi t^{\frac{1-\delta}{2}}\norm{(u_m,f_m)}^2_{L^2_{\mu}}+ \sum_{0\leq k\leq 1}\Big(\frac{\delta}{2} \Big)^{k} \comi t^{\frac{5-\delta}{2}+k} \norm{(\partial_y^k\mathcal{U},\partial_y^k\mathcal{F})}^2_{H^{5,0}_{\mu}}.
\end{equation} 
Accordingly, define  $\mathcal D_\delta=\mathcal D_\delta(t)$ by 
\begin{equation}\label{ddelta}
\begin{aligned}
	\mathcal D_\delta:= &\sum_{m\leq 8}\delta\comi t^{\frac{1-\delta}{2}}\norm{(\partial_yu_m, \partial_y f_m)}^2_{L^2_{\mu}} \\
 &\qquad\qquad+\sum_{0\leq k\leq 1}\Big(\frac{\delta}{2} \Big)^{k+1} \comi t^{\frac{5-\delta}{2}+k} \norm{(\partial_y^{k+1}\mathcal{U},\partial_y^{k+1}\mathcal{F})}^2_{H^{5,0}_{\mu}}.
\end{aligned}
\end{equation}
Recall $u_m,f_m$ and $\mathcal U,\mathcal F$ are defined in \eqref{eq2.1}, \eqref{eqfm} and
\eqref{eq3+}, and the weighted Sobolev spaces are defined  in \eqref{hmkomega}.
\end{definition}

\begin{theorem}[A priori estimate]\label{thm:priori}
Let the initial data $u_{in}$ and $f_{in}$  in \eqref{eq2} satisfy   assumption \eqref{eq4} and the compatibility condition \eqref{comcon}.  Then there exists a small constant $\varepsilon>0$    such that  if 
 \begin{equation}
\mathcal E_\delta(0)  \leq \varepsilon^2
\label{eq2.3}
\end{equation}
 and  if  $  u, f\in L^\infty([0,+\infty[;\  H^{8,0}_\mu\cap H^{5,1}_{\mu_\lambda}) $,  $0\leq \lambda<1,$     is any global solution to    system \eqref{eq2}  satisfying that 
 \begin{equation*} 
 \forall\ t>0,\quad \mathcal E_\delta(t)+\frac{1}{2}\int^{t}_{0}\mathcal D_\delta(s)ds\leq 2\varepsilon^2,
\end{equation*}
then  
\begin{equation}
 \forall\ t>0,\quad \mathcal E_\delta(t)+\frac{1}{2}\int^{t}_{0}\mathcal D_\delta(s)ds\leq \varepsilon^2.
\label{eq2.4}
\end{equation}
\end{theorem}

The proof of Theorem  \ref{thm:priori} is postponed to Sections \ref{SecP}-\ref{SecE2}. By bootstrap principle,    if   $ u, f\in L^\infty([0,+\infty[;\  H^{8,0}_\mu\cap H^{5,1}_{\mu_\lambda})  $ solve the MHD boundary layer system \eqref{eq2},  then  assertion \eqref{eq2.4}  holds true,   provided  the smallness assumption  \eqref{eq2.3} is fulfilled.  

\begin{remark}\label{rmk:lowup} As shown by Lemma \ref{lemu0} in the next section, the weighted Sobolev norm of lower order derivatives of $(u,f)$ can be controlled by that of $(\mathcal{U}, \mathcal F)$, that is, for any $0\leq \lambda<1$ it holds that 
\begin{align*}
  \sum_{0\leq k\leq 1} \comi t^{\frac{5-\delta}{2}+k}\norm{ (\partial_y^ku,  \partial_y^kf)}_{H^{5,0}_{\mu_\lambda }}^2\leq C \sum_{0\leq k\leq 1} \comi t^{\frac{5-\delta}{2}+k}\norm{ (\partial_y^k\mathcal U,  \partial_y^k\mathcal F)}_{H^{5,0}_{\mu }}^2  \leq C \mathcal E_\delta,
 \end{align*}
 and similarly,
 \begin{align*}
  \sum_{0\leq k\leq 1} \comi t^{\frac{5-\delta}{2}+k}\norm{ (\partial_y^{k+1}u,  \partial_y^{k+1}f)}_{H^{5,0}_{\mu_\lambda }}^2\leq   C \mathcal D_\delta, 
 \end{align*}
 where the constant $C$ depends on $\delta $ and $\lambda.$  
 Moreover,  if    $\mathcal E_{\delta}$ is small enough, then it follows from Lemma \ref{lemu8} below that
  \begin{equation*}
  \comi t^{\frac{1-\delta}{2}}\norm{ (u,  f)}_{H^{8,0}_{\mu}}^2  \leq C \mathcal E_\delta    \ \textrm{ and } \ 
    \comi t^{\frac{1-\delta}{2}} \norm{ (\partial_yu,  \partial_yf)}_{H^{8,0}_{\mu}}^2 \leq C \mathcal D_\delta. 
 \end{equation*}
 So the $H^{8,0}_\mu\cap H^{5,1}_{\mu_\lambda}$-norm of $(u,f)$ can be controlled by the energy $\mathcal{E}_\delta.$  
 On the other hand, if $\norm{f_{in}}_{L^\infty}\leq \frac12,$ then we can  verify directly that, by Lemma \ref{lemv} below,  
 \begin{align*}
   \mathcal E_\delta (0)\leq C  \norm{(u_{in}, f_{in})}_{H^{8,0}_{\mu_{in}}}^2\Big(1+ \norm{(u_{in}, f_{in})}^2_{H^{5,1}_{\mu_{in}}}\Big)+C \norm{(u_{in}, f_{in})}^2_{H^{5,1}_{\mu_{in}}}.
 \end{align*}
 So assumption \eqref{smallassumption}  yields the smallness condition \eqref{eq2.3}.
\end{remark}

\subsection{Inequalities with time-dependent weights}  We list in this part some weighted inequalities to be used frequently in the proof of Theorem \ref{thm:priori}.

\begin{lemma}[cf.  Lemma 2.5 in \cite{MR4701733}]\label{lema}
Let $\mu_{\lambda}$ be defined in \eqref{mulambda}, and let $h$ be a function belonging to $H^1$ in $y$ variable, which decays to zero sufficiently fast as $y\to +\infty$. Then for $0\leq \lambda \leq 1$ we have
\begin{equation}\label{ADD1}
    \frac{\lambda^{1\over2}}{  (2 \comi t)^{1\over2}}\norm{h}_{L_{\mu_{\lambda}}^2}\leq \norm{\partial_y h}_{L^2_{\mu_{\lambda}}},   
\end{equation}
and
\begin{align}\label{eq2.6}
   \frac{\lambda^{1\over2}}{  2 \comi t^{1\over2}}\norm{h}_{L_{\mu_{\lambda}}^2}+  \frac{\lambda }{4}\Big\|\frac{y}{\comi t}h\Big\|_{L^2_{\mu_{\lambda}}}\leq  2  \norm{\partial_y h}_{L^2_{\mu_{\lambda}}}.
\end{align}
\end{lemma}

\begin{proof}
	For the sake of completeness we re-present the proof given in    \cite[Lemma 2.5]{MR4701733}.  It suffices to consider the case of $0<\lambda\leq 1.$
	 Using integration by parts yields
\begin{multline}\label{ADD2}
\int^{+\infty}_0h^2 e^{\frac{\lambda y^2}{4\comi t}}dy=-\int^{+\infty}_0y\partial_y\inner{h^2 e^{\frac{\lambda y^2}{4\comi t}}}dy\\
=-2\int^{+\infty}_0yh (\partial_yh) e^{\frac{\lambda y^2}{4\comi t}}dy-\frac{\lambda}{2\comi t}\int^{+\infty}_0y^2h^2 e^{\frac{\lambda y^2}{4\comi t}}dy.
\end{multline}
This, with the fact that
\begin{equation*}
	-2\int^{+\infty}_0yh (\partial_yh) e^{\frac{\lambda y^2}{4\comi t}}dy\leq \frac{\lambda}{2\comi t}\int^{+\infty}_0y^2h^2 e^{\frac{\lambda y^2}{4\comi t}}dy+\frac{2\comi t}{\lambda}\int^{+\infty}_0(\partial_yh)^2 e^{\frac{\lambda y^2}{4\comi t}}dy,
\end{equation*}
yields 
\begin{align*}
 \int^{+\infty}_0h^2 e^{\frac{\lambda y^2}{4\comi t}}dy \leq  \frac{2\comi t}{\lambda}\int^{+\infty}_0(\partial_yh)^2 e^{\frac{\lambda y^2}{4\comi t}}dy.   
\end{align*}
Multiplying  the above estimate by $\frac{\lambda }{2\comi t}$, we obtain   assertion \eqref{ADD1} in Lemma  \ref{lema}.

On the other hand, combining \eqref{ADD2}  with the fact that
\begin{equation*}
	-2\int^{+\infty}_0yh (\partial_yh) e^{\frac{\lambda y^2}{4\comi t}}dy\leq \frac{\lambda}{4\comi t}\int^{+\infty}_0y^2h^2 e^{\frac{\lambda y^2}{4\comi t}}dy+\frac{4\comi t}{\lambda}\int^{+\infty}_0(\partial_yh)^2 e^{\frac{\lambda y^2}{4\comi t}}dy,
\end{equation*}
yields 
\begin{align*}
 \int^{+\infty}_0h^2 e^{\frac{\lambda y^2}{4\comi t}}dy+\frac{\lambda}{4\comi t}\int^{+\infty}_0y^2h^2 e^{\frac{\lambda y^2}{4\comi t}}dy  \leq  \frac{4\comi t}{\lambda}\int^{+\infty}_0(\partial_yh)^2 e^{\frac{\lambda y^2}{4\comi t}}dy.   
\end{align*}
Multiplying  the above estimate by $\frac{\lambda }{4\comi t}$, we obtain the second assertion \eqref{eq2.6}. The proof  of Lemma  \ref{lema} is thus completed.   
\end{proof}

\begin{lemma}\label{lemv}
Let $\mu_{\lambda}$ be defined in  \eqref{mulambda}    and let $h$ be a function belonging to $H^1$ in $y$ variable, which decays to zero sufficiently fast as $y\to +\infty$.  Then \begin{align*}
  \forall \ 0\leq \lambda<1,\quad   \norm{\mu_{\frac{\lambda}{2}}h}_{L_y^{\infty}}\leq C_{\lambda} \comi t^{\frac{1}{4}}\norm{\mu_{\frac{\lambda+1}{4}}\partial_y h}_{L^2_y},
\end{align*}
where the constant  $C_\lambda$  depends only on $\lambda.$
\end{lemma}

\begin{proof} It  just  follows from straightforward verification. In fact,  observe $\mu_\lambda$ is an increasing function of $y$. Then for any $0\leq \lambda <1$  it holds that 
    \begin{align*}
    \norm{\mu_{\frac{\lambda}{2}}h}_{L_y^{\infty}}&=\Big| \mu_{\frac{\lambda}{2}}\int^{+\infty}_y\partial_yhd\tilde y\Big| \leq  \Big| \int^{+\infty}_y\mu_{\frac{\lambda+1}{4}}(\tilde y)\mu_{\frac{\lambda-1}{4}}(\tilde y)\partial_yhd\tilde y\Big|\\     &\leq  \norm{\mu_{\frac{\lambda-1}{4}}}_{L^2_y}\norm{\mu_{\frac{\lambda+1}{4}}\partial_yh}_{L^2_y}\leq C_{\lambda}\comi t^{\frac{1}{4}}\norm{\mu_{\frac{\lambda+1}{4}}\partial_yh}_{L^2_y},
    \end{align*}
  the last inequality using \eqref{integr}.  This completes the proof of  Lemma \ref{lemv}.
\end{proof}

\section{Preliminary  lemmas}\label{SecP}

Before proving the {\it a priori } estimate, we first list     some lemmas  to be used frequently in the proof.   We begin with the treatment of $ \partial_x^mu$ and $ \partial_x^m f $. Although these terms are not involved in the  definition of $\mathcal E_\delta$ and $\mathcal D_\delta$ when $m\geq 1$, as to be seen in Lemmas  \ref{lemu0} and \ref{lemu8}  below,  they can be controlled by the auxiliary functions $u_m$ and $f_m$,   and furthermore by  $\partial_x^m\mathcal U$ and  $\partial_x^m\mathcal F$  if $m\leq 5$.   

  To simplify the notations, throughout the proof  we will use $C_\lambda$   to denote some generic constants  depending  only on the number  $\lambda$ in \eqref{mulambda}. 
 Moreover we use $C$ to denote some  generic constants which may vary from line to line and depend only on the Sobolev embedding constants and $\lambda$, but are independent of $\delta$ and $\varepsilon$.  

\begin{lemma}\label{lemu0}Let $\mu_\lambda$ be defined in \eqref{mulambda}. 
      For any $0\leq \lambda <1$ it holds that 
     \begin{align}\label{eq2.14}
         \norm{\partial_y^k u}_{H^{5,0}_{\mu_{\lambda}}}\leq C_{\lambda}\norm{\partial_y^k \mathcal{U}}_{H^{5,0}_{\mu}},\quad \norm{\partial_y^k  f}_{H^{5,0}_{\mu_{\lambda}}}\leq C_{\lambda}\norm{\partial_y^k \mathcal{F}}_{H^{5,0}_{\mu}},\quad k=0,1,2.
     \end{align} 
\end{lemma}

\begin{proof}
 We proceed to prove \eqref{eq2.14} for $k=0,1$ and $2$. It suffices to derive the estimate on $u$, since the treatment of    $f$ is similar.   

 {\it Step 1). (Estimate on $u$)}. 
By   definition \eqref{eq3} of $\mathcal{U}$   and \eqref{zeromean}, we have
\begin{equation}\label{eq2.15}
	\left\{
\begin{aligned}
&\Big( \partial_y+\frac{y}{2\comi t}\Big)\int^{+\infty}_yud\tilde y=-u+\frac{y}{2\comi t} \int^{+\infty}_yud\tilde y=-\mathcal{U},\\
&\int^{+\infty}_0ud\tilde y=0.
\end{aligned}
\right.
\end{equation}
Solving this ODE system \eqref{eq2.15}, we get
\begin{align*}
\int^{+\infty}_yud\tilde y=-\exp\Big(-\frac{y^2}{4\comi t}\Big)\int^y_0 \exp\Big(\frac{\tilde y^2}{4\comi t}\Big)\mathcal{U}d\tilde y.
\end{align*}
Then we have, for any $0\leq m \leq 5$ and any $0\leq\lambda<1$,
\begin{multline}\label{expu}
\exp\Big(\frac{\lambda y^2}{8 \comi t}\Big)\int^{+\infty}_y\partial_x^mud\tilde y\\
=-\exp\Big(\frac{(\lambda-1) y^2}{8 \comi t}\Big)\int^y_0 \exp\Big(\frac{\tilde y^2-y^2}{8\comi t}\Big) \exp\Big(\frac{\tilde y^2}{8\comi t}\Big)\partial_x^m\mathcal{U}d\tilde y.
\end{multline}
On the other hand, observing 
\begin{align*}
 \forall\ \xi>0,\quad    e^{-\xi^2}\int^{\xi}_0 e^{z^2}dz\leq 2,
\end{align*}
we have
\begin{align*}
   \int^y_0 \exp\Big(\frac{\tilde y^2-y^2}{4\comi t}\Big) d\tilde y =2 \comi t^{1\over2} \exp\Big(-\frac{y^2}{4\comi t}\Big)\int^{\frac{y}{2\sqrt{\comi t}}}_0e^{z^2}dz\leq 4  \comi t^{1\over2}.
\end{align*}
As a result,  in view of \eqref{expu}, 
\begin{align*}
& \mu_{\lambda}(t, y) \Big| \int^{+\infty}_y\partial_x^mud\tilde y\Big |^2= \Big|\exp\Big(\frac{\lambda y^2}{8 \comi t}\Big)\int^{+\infty}_y\partial_x^mud\tilde y\Big |^2 \\
& \leq  \exp\Big(\frac{(\lambda-1) y^2}{4 \comi t}\Big)\bigg[ \int^y_0 \exp\Big(\frac{\tilde y^2-y^2}{4\comi t}\Big) d\tilde y\bigg]  \norm{\partial_x^m\mathcal{U}}_{L^2_{\mu}}^2\\
&\leq   4  \comi t^{1\over2}  \exp\Big(\frac{(\lambda-1) y^2}{4 \comi t}\Big)   \norm{\partial_x^m\mathcal{U}}_{L^2_{\mu}}^2,
\end{align*}
which yields that, for any $0\leq m\leq 5,$
\begin{multline*}
    \Big\|\frac{y}{2\comi t}\int^{+\infty}_y\partial_x^mud\tilde y\Big\|^2_{L^2_{\mu_{\lambda}}}\\\leq  \comi t^{\frac{1}{2}}\int^{+\infty}_0\frac{y^2}{\comi t^2}\exp\Big(\frac{(\lambda-1) y^2}{4 \comi t}\Big)dy\norm{\partial_x^m\mathcal{U}}^2_{L^2_{\mu}}\leq C_{\lambda}\norm{\partial_x^m\mathcal{U}}^2_{L^2_{\mu}},
\end{multline*}
where in the last inequality  we  used the fact that
\begin{align*}
\forall\  0\leq \lambda<1,\quad  \comi t^{\frac{1}{2}}\int^{+\infty}_0\frac{y^2}{\comi t^2}\exp\Big(\frac{(\lambda-1) y^2}{4 \comi t}\Big)dy=   \int^{+\infty}_{0}y^2e^{\frac{(\lambda-1)}{4}y^2}dy\leq C_{\lambda}.
\end{align*}
Thus, using \eqref{eq3} yields 
\begin{equation}\label{eq2.16}
  \forall\ 0\leq \lambda<1,\ \   \norm{u}_{H^{5,0}_{\mu_{\lambda}}} \leq  \Big\|\frac{y}{2\comi t}\int^{+\infty}_yud\tilde y\Big\|_{H^{5,0}_{\mu_{\lambda}}}+ \norm{\mathcal{U}}_{H^{5,0}_{\mu_{\lambda}}}\leq C_{\lambda} \norm{\mathcal{U}}_{H^{5,0}_{\mu}}.  
\end{equation}
 
  {\it Step 2). (Estimate on $\partial_yu$)}. We apply $\partial_y$ to  the representation of $\mathcal{U}$  in \eqref{eq3}  to get
\begin{equation*} 
    \partial_yu=\partial_y\mathcal{U}+\frac{1}{2\comi t}\int^{+\infty}_y u d\tilde y-\frac{y}{2\comi t}u.
\end{equation*}
This yields
\begin{equation}\label{pyu}
   \norm{\partial_yu}_{H^{5,0}_{\mu_{\lambda}}} \leq  \norm{\partial_y\mathcal{U}}_{H^{5,0}_{\mu_{\lambda}}}+\frac{1}{2\comi t} \Big\|\int^{+\infty}_y u d\tilde y\Big\|_{H^{5,0}_{\mu_{\lambda}}}+\frac{1}{2\comi t}\norm{ y u}_{H^{5,0}_{\mu_{\lambda}}}.
\end{equation}
For the second term on the right-hand side, we use Lemma \ref{lema} and  \eqref{eq2.16} to get
\begin{align*}
\frac{1}{2\comi t} \Big\|\int^{+\infty}_y u d\tilde y\Big\|_{H^{5,0}_{\mu_{\lambda}}} \leq \frac{C_{\lambda}}{\comi t^{1\over2}}\norm{u}_{H^{5,0}_{\mu_{\lambda}}} \leq  \frac{C_{\lambda}}{\comi t^{1\over2}}\norm{\mathcal{U}}^2_{H^{5,0}_{\mu}}\leq C_{\lambda}\norm{\partial_y\mathcal{U}}^2_{H^{5,0}_{\mu}} .
\end{align*}
Moreover,   for any $0\leq \lambda <1$, we use   the fact that 
$\norm{y^2e^{-a y^2}}_{L^\infty}\leq C_a$ for any constant $a>0,$ to compute 
\begin{equation}\label{compu+}
\frac{1}{2\comi t}\norm{ y u}_{H^{5,0}_{\mu_{\lambda}}}\leq \frac{C_{\lambda}}{\comi t^{1\over2 }}\norm{u}_{H^{5,0}_{\mu_{\frac{\lambda+1}{2}}}} \leq   \frac{C_{\lambda}}{\comi t^{1\over2}}\norm{\mathcal{U}}_{H^{5,0}_{\mu}}\leq C_{\lambda}\norm{\partial_y\mathcal{U}}_{H^{5,0}_{\mu}},
\end{equation}
the second inequality using \eqref{eq2.16} and the last inequality following from Lemma \ref{lema}.  Now we 
combine  the above estimates with \eqref{pyu}  to conclude that
\begin{align}\label{eq2.17}
 \forall\ 0\leq \lambda<1,\quad      \norm{\partial_yu}_{H^{5,0}_{\mu_{\lambda}}} \leq  C_{\lambda}\norm{\partial_y\mathcal{U}}_{H^{5,0}_{\mu}}. 
\end{align}

 {\it Step 3). (Estimate on $\partial_y^2u$)}.
  We apply $\partial_y^2$ to the representation of $\mathcal{U}$ in \eqref{eq3} to get
\begin{align*} 
    \partial_y^2u=\partial_y^2\mathcal{U}-\frac{1}{\comi t}u-\frac{y}{2\comi t}\partial_yu.
\end{align*}
Using Lemma \ref{lema} and  \eqref{eq2.16} again  yields that, for any $0\leq \lambda<1,$
 \begin{align*}
 \frac{1}{\comi t}\norm{u}_{H^{5,0}_{\mu_{\lambda}}}\leq  \frac{C_\lambda}{\comi t}\norm{\mathcal{U}}_{H^{5,0}_{\mu}} \leq \frac{C_{\lambda}}{\comi t^{\frac{1}{2}}}\norm{\partial_y\mathcal{U}}_{H^{5,0}_{\mu}}\leq C_{\lambda}\norm{\partial_y^2\mathcal{U}}_{H^{5,0}_{\mu}}.  
\end{align*}
Moreover, repeating the argument in \eqref{compu+} and then using   \eqref{eq2.17} we  have  
\begin{align*}
\frac{1}{2\comi t}\norm{y\partial_yu}_{H^{5,0}_{\mu_{\lambda}}}\leq \frac{C_{\lambda}}{\comi t^{1\over2 }}\norm{\partial_yu}_{H^{5,0}_{\mu_{\frac{\lambda+1}{2}}}}\leq   \frac{C_{\lambda}}{\comi t^{1\over2 }}\norm{\partial_y\mathcal{U}}_{H^{5,0}_{\mu}}\leq C_{\lambda}\norm{\partial_y^2\mathcal{U}}_{H^{5,0}_{\mu}},
\end{align*}
the last inequality following from Lemma \ref{lema}.
As a result, we combine the above estimates to conclude that, for any $0\leq \lambda<1,$
\begin{align*}
    \norm{\partial_y^2u}_{H^{5,0}_{\mu_{\lambda}}} \leq  \norm{\partial_y^2\mathcal{U}}_{H^{5,0}_{\mu}}+ \frac{1}{\comi t}\norm{u}_{H^{5,0}_{\mu_{\lambda}}}+\frac{1}{2\comi t}\norm{y\partial_yu}_{H^{5,0}_{\mu_{\lambda}}}\leq C_{\lambda}\norm{\partial_y^2\mathcal{U}}_{H^{5,0}_{\mu}}. 
\end{align*}
This together with \eqref{eq2.16} and \eqref{eq2.17} yields the desired estimate on $u$ in \eqref{eq2.14}. The proof of Lemma \ref{lemu0} is thus completed. 
\end{proof}

As an immediate consequence of Lemma \ref{lemu0} and definitions \eqref{edelta} and \eqref{ddelta} of $\mathcal E_\delta$ and $\mathcal D_\delta$,  we have the following 

\begin{corollary}
	\label{coro}
	For any $0\leq \lambda<1$ and  any $t\geq 0$ it holds that
	\begin{equation*}
	\sum_{0\leq k\leq 1} \Big(\frac{\delta}{2} \Big)^{\frac{k}{2}} \comi t^{\frac{5-\delta}{4}+\frac{k}{2}} \norm{(\partial_y^ku, \partial_y^kf)}_{H^{5,0}_{\mu_\lambda}}   \leq C_{\lambda} \big(\mathcal E_\delta(t)\big)^{1\over2}, 
	\end{equation*}
	and 
	\begin{equation*}
		  \sum_{0\leq k\leq 1} \Big(\frac{\delta}{2} \Big)^{\frac{k+1}{2}} \comi t^{\frac{5-\delta}{4}+\frac{k}{2}} \norm{(\partial_y^{k+1}u, \partial_y^{k+1}f)}_{H^{5,0}_{\mu_\lambda}} \leq C_{\lambda}\big(\mathcal D_\delta(t)\big)^{1\over2}.  	\end{equation*}
Thus if 
	 \begin{equation*}
\sup_{t\geq 0}\mathcal E_{\delta}(t)  \leq 2\varepsilon^2,
\end{equation*}
then
	\begin{equation}\label{asslow}
\forall\ 0\leq \lambda<1, \ \forall \ t\geq 0, \quad 	 	\sum_{0\leq k\leq 1} \Big(\frac{\delta}{2} \Big)^{\frac{k}{2}} \comi t^{\frac{5-\delta}{4}+\frac{k}{2}} \norm{(\partial_y^ku, \partial_y^kf)}_{H^{5,0}_{\mu_\lambda}}  \leq \varepsilon C_{\lambda}.
	\end{equation}
\end{corollary}

The following   lemma  is devoted to controlling $(\partial_x^m u, \partial_x^m f)$  in terms of   $(u_m,  f_m)$ under the smallness assumption.   

 \begin{lemma}\label{lemu8}
There exists a small   constant $\varepsilon>0$  such that if 
 \begin{equation*}
 	 \sup_{t\geq 0}\mathcal E_\delta(t)\leq 2\varepsilon^2,
\end{equation*}
then for each $1\leq m\leq 8$ it holds that
         \begin{equation}\label{ceps}
        \comi t^{\frac{1-\delta}{4}} \norm{(\partial_x^mu, \partial_x^mf)}_{L^2_{\mu}} \leq  C\comi t^{\frac{1-\delta}{4}} \norm{(u_m,  f_m)}_{L^2_{\mu}}  \leq C  \big( \mathcal{E}_\delta(t)\big)^{1\over 2} \leq C \varepsilon,
        \end{equation}
        and that
        \begin{align*}  
         \delta^{1\over 2}\comi t^{\frac{1-\delta}{4}} \norm{(\partial_x^m\partial_yu, \partial_x^m\partial_yf)}_{L^2_{\mu}}   \leq C  \big( \mathcal{D}_\delta(t)\big)^{1\over 2}.
         \end{align*}
 \end{lemma}
 
 \begin{proof} From Lemma \ref{lemv} and  assertion \eqref{asslow} in Corollary \ref{coro}, it follows that       \begin{align*}
         \norm{f}_{L^{\infty}}\leq C \comi t^{1\over 4} \norm{\partial_y f}_{H^{1,0}_{\mu_{1\over2}}}\leq C\varepsilon\delta^{-\frac{1}{2}}.
     \end{align*}
     Thus we can find a small $\varepsilon>0$ such that $C\varepsilon\delta^{-\frac{1}{2}}\leq {1\over2}$. This yields 
     \begin{equation}\label{ublb}
\forall\  (t,x,y)\in [0,+\infty[ \,  \times \mathbb R_+^2, \quad     \frac12  \leq  1+ f(t,x,y)\leq  \frac{3}{2}.
     \end{equation}
Thus, by   definition \eqref{eq2.1} of  $u_m,$ 
\begin{align*}
  \forall\ 1\leq m\leq 8,\quad      \norm{\partial_x^m u}_{L^2_{\mu}} &\leq       \norm{u_m}_{L^2_{\mu}} + 2\norm{(\partial_x^{m-1}g)\partial_yu}_{L^2_{\mu}}.
\end{align*} 
   On the other hand, we use Lemma \ref{lemv} to compute
   \begin{equation}\label{compu}
   \begin{aligned}
     \norm{(\partial_x^{m-1}g)\partial_yu}_{L^2_{\mu}} & \leq \norm{\mu_{1\over 4}\partial_x^{m-1}g}_{L_x^2 L_y^\infty}  \norm{\mu_{1\over 4} \partial_yu}_{L_x^\infty L_y^2}\\
       &    \leq C       \comi t^{\frac{1}{4}}   \norm{\partial_x^m f}_{L^2_{\mu}}\norm{\partial_yu}_{H^{1,0}_{\mu_{1\over2}}} \leq C  \varepsilon  \delta^{-\frac{1}{2}}    \norm{\partial_x^m f}_{L^2_{\mu}},
       \end{aligned}
   \end{equation} 
   the last inequality  using   \eqref{asslow} again.
     Combining the above estimates yields
   \begin{align*}
        \norm{\partial_x^m u}_{L^2_{\mu}} &\leq      \norm{u_m}_{L^2_{\mu}} +C \varepsilon \delta^{-\frac{1}{2}}    \norm{\partial_x^m f }_{L^2_{\mu}} .
   \end{align*}  
 Similarly, we have
    \begin{align*}
        \norm{\partial_x^m f }_{L^2_{\mu}}\leq    \norm{f_m}_{L^2_{\mu}} +C\varepsilon \delta^{-\frac{1}{2}}    \norm{\partial_x^mu}_{L^2_{\mu}}.
   \end{align*}
 As a result, by choosing $\varepsilon$ such that $\varepsilon \delta^{-\frac{1}{2}} $ is small  sufficiently,   we have
     \begin{align*}
  \forall\ 1\leq m\leq 8,\quad \norm{(\partial_x^mu, \partial_x^m f )}_{L^2_{\mu}}  \leq   C \norm{(u_m, f_m)}_{L^2_{\mu}},   
    \end{align*}
   which with definition \eqref{edelta} of $\mathcal E_\delta$ yields the first assertion in Lemma \ref{lemu8}.  
   
   It remains to prove the second assertion.  By  definition \eqref{eq2.1} of  $u_m$, 
\begin{equation}\label{yum}
	\partial_x^m\partial_yu = \partial_y u_m+\frac{\partial_y u}{1+f}\partial_x^mf  - \frac{\partial_y^2 u}{1+f}  \partial_x^{m-1}g+\frac{(\partial_y  u)\partial_yf}{(1+f)^2}  \partial_x^{m-1}g.
\end{equation}
Following the argument in \eqref{compu}  we can verify that
\begin{align*}
 	\Big\|\frac{\partial_y^2u}{1+f}\partial_x^{m-1}g\Big\|_{L^2_{\mu}} \leq C  \comi t^{\frac{1}{4}}   \norm{\partial_x^m f}_{L^2_{\mu}}\norm{\partial_y^2u}_{H^{1,0}_{\mu_{1\over2}}}\leq C\varepsilon \comi t^{\frac{\delta}{4}}  \norm{\partial_y^2\mathcal U}_{H^{1,0}_{\mu}}, 
\end{align*}
the last inequality using  Lemma  \ref{lemu0} and the fact that 
\begin{equation}\label{t14}
	 \comi t^{\frac{1}{4}}   \norm{\partial_x^m f}_{L^2_{\mu}} \leq    \comi t^{\frac{ \delta}{4}}    \comi t^{\frac{1-\delta}{4}}  \norm{\partial_x^m f}_{L^2_{\mu}}  \leq C \varepsilon   \comi t^{\frac{ \delta}{4}}  
\end{equation}
due to \eqref{ceps}. 
Similarly, 
\begin{align*}
&	\Big\|\frac{\partial_yu}{1+f}\partial_x^m f \Big\|_{L^2_{\mu}} \leq 2 \norm{\mu_{1\over4}\partial_yu}_{L^\infty} \norm{\partial_x^mf}_{L^2_{\mu_{1\over 2}}}\leq C\comi t^{1\over 4}\norm{\partial_y^2 u}_{H^{1,0}_{\mu_{3\over 4}}}\norm{\partial_x^mf}_{L^2_{\mu}}\\
&\qquad \leq  C\comi t^{1\over 4}\norm{\partial_y^2 \mathcal U}_{H^{1,0}_{\mu }}\norm{\partial_x^mf}_{L^2_{\mu}}\leq C\varepsilon \comi t^{\frac{\delta}{4}}  \norm{\partial_y^2\mathcal U}_{H^{1,0}_{\mu}} .\end{align*}
Moreover,    using Lemma \ref{lemv} and estimate \eqref{asslow}    as well as \eqref{t14},
\begin{align*}
      &\Big\|\frac{ (\partial_yu)\partial_yf}{(1+f)^2}\partial_x^{m-1}g \Big\|_{L^2_{\mu}} \leq 4 \norm{ \mu_{1\over 8}\partial_y u}_{L^\infty} \norm{ \mu_{1\over 8}  \partial_yf}_{L_x^\infty L_y^2 }\norm{  \mu_{1\over 4} \partial_x^{m-1}g}_{L_x^2 L^\infty_y }   \\
      &\qquad \leq C\comi t^{1\over4}\norm{ \partial_y^2 \mathcal U}_{H^{1,0}_\mu} \norm{ \partial_y  f}_{H^{1,0}_{\mu_{1\over4}}}\Big( \comi t^{1\over4}\norm{   \partial_x^{m}f}_{L_\mu^2}\Big)\leq C\varepsilon^2\delta^{-\frac12}\norm{ \partial_y^2 \mathcal U}_{H^{1,0}_\mu}.
   \end{align*} 
Combining the above estimates with \eqref{yum} yields that, for any $1\leq m\leq 8,$
\begin{equation*}
	\norm{\partial_x^m \partial_yu }_{L^2_\mu}\leq \norm{  \partial_yu_m }_{L^2_\mu}+C\varepsilon\delta^{-{ 1\over2}} \comi t^{\frac{\delta}{4}}     \norm{\partial_y^2\mathcal U}_{H^{1,0}_{\mu}},
\end{equation*}
and thus, by definition \eqref{ddelta} of $\mathcal D_\delta$ and the fact that $\eps\delta^{-1}$ is small, 
\begin{equation*}
	\delta^{1\over 2}\comi t^{\frac{1-\delta}{4}}\norm{\partial_x^m \partial_yu }_{L^2_\mu}\leq \big(\mathcal D_\delta(t)\big)^{1\over2}+C\eps\delta^{-1} \big(\delta\comi t^{\frac{7-\delta}{4}}\norm{\partial_y^2\mathcal U}_{H^{1,0}_{\mu}}\big)\leq C  \big(\mathcal D_\delta(t)\big)^{1\over2}.
\end{equation*}
Similarly for $\norm{\partial_x^m \partial_yf}_{L^2_\mu}.$ Thus the second  assertion  in Lemma  \ref{lemu8} follows. The proof   is  thus completed. 
 \end{proof}

 \begin{lemma}
 	[Technical lemma]\label{techlema}
 	 Let $\varphi\in H^{0,1}_\mu$ and $\psi\in L_\mu ^2$ satisfy that 	\begin{equation}\label{varpsi}
 		(\partial_t -\partial_y^2) \varphi= \psi,\quad \varphi \partial_y\varphi|_{y=0}=0. 
 	\end{equation}
 	Then we have
 	\begin{equation}\label{as1}
		\frac{d}{dt}\norm{\varphi}^2_{L^2_{\mu}}+\frac{1-\delta}{2\comi t}\norm{\varphi}^2_{L^2_{\mu}}+\delta\norm{ \partial_y\varphi}^2_{L^2_{\mu}}\leq 2 \big( \psi,\  \varphi \big)_{L^2_{\mu}},
	\end{equation}
	and  
\begin{equation}\label{as2}
		\frac{d}{dt}\Big(\comi t^{\frac{1-\delta}{2}}\norm{\varphi}^2_{L^2_{\mu}} \Big)+\delta \comi t^{\frac{1-\delta}{2}}\norm{ \partial_y\varphi}^2_{L^2_{\mu}}\leq 2 \comi t^{\frac{1-\delta}{2}}\big( \psi,\  \varphi \big)_{L^2_{\mu}}.
	\end{equation}
 \end{lemma}
 
\begin{proof}
It suffices to prove the first assertion \eqref{as1}, since the second one \eqref{as2} will follow if we multiply \eqref{as1} by $\comi t^{\frac{1-\delta}{2}}$ and then use the fact that
\begin{align*}
\frac{d}{dt} \comi t^{\frac{1-\delta}{2}} =\comi t^{\frac{1-\delta}{2}} \frac{1-\delta}{2\comi t}.
\end{align*}
By virtue of the fact that
\begin{align*}
    \partial_t\mu =-\frac{y^2}{4\comi t^2}\mu,\quad \partial_y^2\mu=\frac{1}{2\comi t}\mu +\frac{y^2}{4\comi t^2}\mu,
\end{align*}
we compute
\begin{align*}
   2 \int_{\mathbb{R}^2_{+}}(\partial_t \varphi)\varphi \mu dx dy&= \frac{d}{dt}\norm{\varphi}^2_{L^2_{\mu}}-  \int_{\mathbb{R}^2_{+}} \varphi^2 (\partial_t\mu)  dxdy = \frac{d}{dt}\norm{\varphi}^2_{L^2_{\mu}}+\frac{ \norm{y\varphi}_{L_\mu^2} ^2}{4\comi t^2},
\end{align*}
and
\begin{align*}
  &-2\int_{\mathbb{R}^2_{+}}(\partial_y^2\varphi)\varphi \mu dxdy =2\norm{ \partial_y\varphi}^2_{L^2_{\mu}}+2\int_{\mathbb{R}^2_{+}}(\partial_y \varphi)\varphi \partial_y\mu dxdy \\
&\qquad  =2\norm{ \partial_y\varphi}^2_{L^2_{\mu}} - \int_{\mathbb{R}^2_{+}}\varphi^2\partial_y^2\mu dxdy   =2\norm{ \partial_y\varphi}^2_{L^2_{\mu}}-\frac{1}{2\comi t} \norm{\varphi}_{L_\mu^2} ^2-\frac{\norm{y\varphi}_{L_\mu^2} ^2}{4\comi t^2} .
  \end{align*}
As a result, we take $L_\mu^2$-product with $2\varphi$ in \eqref{varpsi} and then combine the two equations above to get
\begin{equation}\label{dtvarphi}
	 \frac{d}{dt}\norm{\varphi}^2_{L^2_{\mu}}+ 2\norm{ \partial_y\varphi}^2_{L^2_{\mu}}-\frac{1}{2\comi t}\norm{\varphi}^2_{L^2_{\mu}}= 2 \inner{\psi,\  \varphi}_{L^2_{\mu}}. 
\end{equation}
Moreover,  it follows from \eqref{ADD1} in Lemma \ref{lema} that
\begin{align*}
    \frac{1}{2\comi t}\norm{\varphi}^2_{L^2_{\mu}}\leq \norm{ \partial_y\varphi}^2_{L^2_{\mu}},
\end{align*}
and thus
\begin{align*}
2\norm{ \partial_y\varphi}^2_{L^2_{\mu}}= \delta \norm{ \partial_y\varphi}^2_{L^2_{\mu}}+\big (2- \delta  \big )\norm{ \partial_y\varphi}^2_{L^2_{\mu}} \geq   \delta  \norm{ \partial_y\varphi}^2_{L^2_{\mu}}+\frac{2-\delta}{2\comi t}\norm{\varphi}^2_{L^2_{\mu}}.
\end{align*}
Substituting the above estimate into \eqref{dtvarphi} yields the first assertion  in Lemma \ref{techlema}. The proof is thus completed.  
\end{proof}

\begin{remark}\label{rmkest}
	 If $\varphi\in H^{0,1}_\mu$ and $\psi\in L_\mu ^2$ satisfy that
	\begin{equation*} 
 		(\partial_t -\partial_y^2) \varphi+\frac{1}{\comi t}\varphi= \psi,\quad \varphi\partial_y\varphi|_{y=0}=0, 
 	\end{equation*}
 	then it follows from  \eqref{as1} in Lemma \ref{techlema} that
 	\begin{equation*}
 	\frac{d}{dt}\norm{\varphi}^2_{L^2_{\mu}}+\frac{5-\delta}{2\comi t}\norm{\varphi}^2_{L^2_{\mu}}+\delta\norm{ \partial_y\varphi}^2_{L^2_{\mu}}\leq 2 \big( \psi,\  \varphi \big)_{L^2_{\mu}}.
 	\end{equation*}
 Thus, similar to \eqref{as2}, it follows from the above estimate that
 \begin{equation*}
 	\frac{d}{dt}\Big(\comi t^{\frac{5-\delta}{2}}\norm{\varphi}^2_{L^2_{\mu}}\Big)+ \delta \comi t^{\frac{5-\delta}{2}}\norm{ \partial_y\varphi}^2_{L^2_{\mu}}\leq 2\comi t^{\frac{5-\delta}{2}} \big( \psi,\  \varphi \big)_{L^2_{\mu}}.
 \end{equation*}
\end{remark}

\section{Estimate on the auxiliary functions $u_m$ and $f_m$}\label{SecE1}

  This part is devoted to dealing with the terms $(u_m, f_m)$ in definitions \eqref{edelta} and \eqref{ddelta} of $\mathcal E_\delta $ and $\mathcal D_\delta$.    The main result can be stated as follows.  
  
  \begin{proposition}\label{propo2}
  Let $\varepsilon$ be the small constant given in Lemma \ref{lemu8}. Suppose  
   \begin{equation}\label{sassm}
 	 \forall\ t>0,\quad \mathcal E_\delta(t)+\frac{1}{2}\int_0^t\mathcal D_\delta (s) ds \leq 2\varepsilon^2.
\end{equation}
 Then 
  \begin{equation}\label{eq3.9}
      \begin{aligned}
          \frac{d}{dt} \sum_{ m\leq 8}  \comi t^{\frac{1-\delta}{2}}\norm{(u_m,f_m)}^2_{L^2_{\mu}}+\delta \sum_{ m\leq 8} \comi t^{\frac{1-\delta}{2}}\norm{(\partial_yu_m,\partial_yf_m)}^2_{L^2_{\mu}} \leq C\varepsilon\delta^{-2} \mathcal{D}_\delta.
      \end{aligned}
  \end{equation}
\end{proposition}

\begin{proof}  It suffices to consider the case of $ 1\leq m\leq 8$, since the energy  estimate for $(u_m, f_m)$ with $m=0$ is straightforward. 
For each fixed  $ 1\leq m\leq 8$,  the auxiliary functions  $u_m, f_m$ solve the following equations (see Appendix \ref{appumfm} for the derivation in detail):
\begin{equation}\label{dix3++}
\left\{
    \begin{aligned}
&\inner{\partial_t+u\partial_x+v\partial_y-\partial_y^2}u_m-\big((1+f)\partial_x+g\partial_y\big)f_m =\sum_{1\leq j\leq 3}S_{m,j},\\
&\inner{\partial_t+u\partial_x+v\partial_y-\partial_y^2}f_m-\big((1+f)\partial_x+g\partial_y\big)u_m =\sum_{1\leq j\leq 3}\tilde S_{m,j},
\end{aligned}
\right.
\end{equation}
where  $S_{m,j}$ and $\tilde S_{m,j}$ are given in \eqref{sm1}-\eqref{sm3t}. It suffices to perform estimates for the first evolution equation in \eqref{dix3++}.  To do so we recall  the representations of $S_{m,j}$ given in \eqref{sm1}-\eqref{sm3},  that is,
\begin{multline}\label{nsm1}
	S_{m,1} =\sum^m_{k=1}\binom{m}{k}\left[(\partial_x^k f )\partial_x^{m-k+1} f -(\partial_x^ku)\partial_x^{m-k+1}u\right]\\
 +\sum^{m-1}_{k=1}\binom{m}{k}\left[(\partial_x^kg)\partial_x^{m-k}\partial_y f -(\partial_x^kv)\partial_x^{m-k}\partial_yu\right],
\end{multline}
and 
\begin{multline}\label{nsm2}
	S_{m, 2} = \frac{\partial_yu}{1+f}\sum^{m-1}_{k=1}\binom{m-1}{k}\left[(\partial_x^k f )\partial_x^{m-k}v-(\partial_x^kg)\partial_x^{m-k}u\right]\\
 -\frac{\partial_yu}{1+f}\sum^{m-1}_{k=1}\binom{m-1}{k}\left[(\partial_x^ku)\partial_x^{m-k}g-(\partial_x^kv)\partial_x^{m-k} f \right],
\end{multline}
and
\begin{equation}\label{nsm3}
	\begin{aligned}
		S_{m,3}&= \left[\frac{g\partial_y f }{1+f}+2\partial_y\inner{\frac{\partial_yu}{1+f}}\right]\partial_x^m f -\frac{g\partial_yu}{1+f}\partial_x^mu-2\frac{(\partial_y f )^2\partial_yu}{(1+f)^3} \partial_x^{m-1}g\\
	&\quad+\left[ \frac{(\partial_y f )\partial_x f -(\partial_yu)\partial_xu}{1+f}+\frac{g(\partial_y f )^2-g(\partial_yu)^2+2(\partial_y f )\partial_y^2u}{(1+f) ^2}\right]\partial_x^{m-1}g.
	\end{aligned}
\end{equation} 
Applying Lemma \ref{techlema} to the first equation in \eqref{dix3++}, we have
 \begin{equation}\label{eq3.12}
\begin{aligned}
&\frac{d}{dt}\inner{\comi t^{\frac{1-\delta}{2}}\norm{u_m}^2_{L^2_{\mu}}}+\delta\comi t^{\frac{1-\delta}{2}} \norm{\partial_yu_m}^2_{L^2_{\mu}} \\
&\leq   2\comi t^{\frac{1-\delta}{2}}  \Big(\big[(1+f)\partial_x+g\partial_y\big] f_m,\ u_m\Big)_{L_\mu^2}+  2\comi t^{\frac{1-\delta}{2}}  \norm{u_m}_{L^2_{\mu}} \sum_{1\leq j\leq 3} \norm{S_{m,j}}_{L^2_{\mu}}\\
&\quad- 2\comi t^{\frac{1-\delta}{2}}  \Big((u\partial_x+v\partial_y)u_m,\ u_m\Big)_{L_\mu^2}.
\end{aligned}
\end{equation}
We claim that  
\begin{equation}\label{upsmj}
\begin{aligned}
  \sum_{1\leq j\leq 3}\norm{S_{m,j}}_{L_\mu^2} 
  \leq   C\comi t^{-\frac{1}{2}}\Big(\varepsilon \delta^{-\frac{1}{2}} \norm{(\partial_yu, \partial_yf)}_{H^{8,0}_{\mu}}+   \varepsilon \norm{\partial_y^2\mathcal{U}}_{H^{5,0}_{\mu}} + \delta^{-2} \mathcal D_\delta\Big),
	\end{aligned}
\end{equation} 
with the proof postponed. 
Estimate \eqref{upsmj} with Lemma \ref{lema} yields 
\begin{align*}
	&\comi t^{\frac{1-\delta}{2}}  \norm{u_m}_{L^2_{\mu}} \sum_{1\leq j\leq 3} \norm{S_{m,j}}_{L^2_{\mu}}\\
	&\leq C\comi t^{\frac{1-\delta}{ 2}} \norm{\partial_yu_m}_{L^2_{\mu}}\Big(\varepsilon\delta^{-\frac{1}{2}} \norm{(\partial_yu, \partial_yf)}_{H^{8,0}_{\mu}}+\varepsilon\norm{\partial_y^2\mathcal{U}}_{H^{5,0}_{\mu}}\Big)\\
    &\quad +  C\delta^{-2}\comi t^{-\frac{\delta}{2}}  \norm{u_m}_{L^2_{\mu}} \mathcal D_\delta \\
	& \leq C\varepsilon \delta^{-2}\mathcal D_\delta,
\end{align*}
the last inequality using definitions \eqref{edelta} and \eqref{ddelta} of $\mathcal E_\delta$ and $\mathcal D_\delta$, Lemma \ref{lemu8} and the smallness assumption \eqref{sassm}.  For the last term on the right-hand side of \eqref{eq3.12}, integrating by parts and then using the  divergence-free condition $\partial_xu+\partial_yv=0$, we have
\begin{equation}\label{Divergence}
\begin{aligned}
 &\left|\comi t^{\frac{1-\delta}{2}}  \Big((u\partial_x+v\partial_y)u_m,\ u_m\Big)_{L_\mu^2}\right|=\frac{1}{2}\comi t^{\frac{1-\delta}{2}}\left|\int_{\mathbb{R}^2_{+}}vu_m^2\partial_y\big(e^{\frac{y^2}{4\comi t}}\big)dxdy\right|\\
 &\leq C\comi t^{\frac{1-\delta}{2}}\norm{v}_{L^{\infty}}\norm{u_m}_{L^2_{\mu}}\Big\|\frac{y}{\comi t}u_m\Big\|_{L^2_{\mu}}\leq  C\comi t^{\frac{1-\delta}{2}+\frac14}  \norm{u}_{H^{2,0}_{\mu_{1\over2}}}\norm{u_m}_{L^2_{\mu}}\Big\|\frac{y}{\comi t}u_m\Big\|_{L^2_{\mu}} \\
 &\leq C\varepsilon \comi t^{\frac{1-\delta}{2}+\frac14}\comi t^{-\frac{5-\delta}{4}}\norm{u_m}_{L^2_{\mu}}\Big\|\frac{y}{\comi t}u_m\Big\|_{L^2_{\mu}} \leq C\varepsilon \comi t^{\frac{1-\delta}{2}}\norm{\partial_yu_m}_{L^2_{\mu}}^2\leq C\varepsilon\delta^{-1}\mathcal D_\delta,
\end{aligned}
\end{equation}
where in the second inequality we used Lemma \ref{lemv}   and in the last line we   used \eqref{asslow} and   \eqref{eq2.6}   as well as definition \eqref{ddelta} of $\mathcal D_\delta.$ 
Substituting the above estimates into \eqref{eq3.12} yields  
\begin{align*}
	&\frac{d}{dt}\inner{\comi t^{\frac{1-\delta}{2}}\norm{u_m}^2_{L^2_{\mu}}}+\delta\comi t^{\frac{1-\delta}{2}} \norm{\partial_yu_m}^2_{L^2_{\mu}} \\
 &\qquad\qquad\leq   2 \comi t^{\frac{1-\delta}{2}}  \Big(\big[(1+f)\partial_x+g\partial_y\big] f_m,\ u_m\Big)_{L_\mu^2}+C\varepsilon \delta^{-2}\mathcal D_\delta.
\end{align*} 
Similarly, repeating the argument above for the second equation in \eqref{dix3++} we have
\begin{align*}
&\frac{d}{dt}\inner{\comi t^{\frac{1-\delta}{2}}\norm{f_m}^2_{L^2_{\mu}}}+\delta\comi t^{\frac{1-\delta}{2}} \norm{\partial_yf_m}^2_{L^2_{\mu}} \\
& \qquad\qquad \leq   2 \comi t^{\frac{1-\delta}{2}}  \Big(\big[(1+f)\partial_x + g\partial_y\big]u_m,\ f_m\Big)_{L^2_{\mu}}+C\varepsilon \delta^{-2}\mathcal D_\delta.
\end{align*}
On the other hand, by a similar argument to that in  \eqref{Divergence}, we have
\begin{multline*}
\comi t^{\frac{1-\delta}{2}} \Big(\big[(1+f)\partial_x + g\partial_y\big ] f_m,\ u_m\Big)_{L^2_{\mu}}+\comi t^{\frac{1-\delta}{2}} \Big(\big [(1+f)\partial_x + g\partial_y\big ]u_m,\ f_m\Big)_{L^2_{\mu}}\\
=-\comi t^{\frac{1-\delta}{2}} \int_{\mathbb{R}^2_{+}}gf_mu_m\partial_y\big(e^{\frac{y^2}{4\comi t}}\big)dxdy\leq C\varepsilon\delta^{-1}\mathcal D_\delta.
\end{multline*}
As a result, combining the above estimates yields that, for any $1\leq m\leq 8, $ 
\begin{equation*}
 \frac{d}{dt}\inner{\comi t^{\frac{1-\delta}{2}}\norm{(u_m, f_m)}^2_{L^2_{\mu}}}+\delta\comi t^{\frac{1-\delta}{2}} \norm{(\partial_yu_m, \partial_yf_m)}^2_{L^2_{\mu}}
 \leq   C\varepsilon \delta^{-2}\mathcal D_\delta.   
\end{equation*}
Note the above estimate still holds true for $m=0$ by performing  straightforward energy estimate for system \eqref{eq2}.   Thus
 assertion \eqref{eq3.9} in Proposition \ref{propo2} follows. 
 
 So to complete the proof of Proposition \ref{propo2}, 
it remains to prove assertion \eqref{upsmj}, and we proceed through the following three steps.  In the following discussion, 
let $1\leq m\leq 8$ be fixed. 
 
    {\it Step  1  (Estimate on $S_{m,1}$)}. 
For any $k$ with $1\leq k\leq [\frac{m}{2}]\leq 4,$ we use Sobolev inequality   to compute  
\begin{equation*}
\begin{aligned}
	& \norm{(\partial_x^k f )\partial_x^{m+1-k} f }_{L^2_{\mu}} \leq  \norm{\mu_{\frac{1}{4}}\partial_x^k f }_{L_x^{\infty}L_y^2}\norm{\mu_{\frac{1}{4}}\partial_x^{m+1-k} f }_{L_x^2L_y^{\infty}}   \\
	&  \leq C\comi t^{\frac14}\norm{ f }_{H^{5,0}_{\mu_{\frac{1}{2}}}} \norm{ \partial_y f }_{H^{8,0}_{\mu}}\leq C  \varepsilon   \comi t^{\frac{1}{4}-\frac{5-\delta}{4} }\norm{\partial_y f }_{H^{8,0}_{\mu}}\leq C\varepsilon\comi t^{-\frac{1}{2}}\norm{\partial_y f }_{H^{8,0}_{\mu}}, 
	\end{aligned}
\end{equation*}
  the last line using  Lemma \ref{lemv} as well as \eqref{asslow}.  Similarly, for any  $k$ with $[\frac{m}{2}]+1\leq k\leq m\leq 8,$  we have  $m+1-k\leq 4$ and  thus 
  \begin{align*}
    \norm{(\partial_x^k f )\partial_x^{m+1-k} f }_{L^2_{\mu}}&\leq  \norm{\mu_{\frac{1}{4}}\partial_x^k f }_{L_x^2L_y^{\infty}}\norm{\mu_{\frac{1}{4}}\partial_x^{m+1-k} f }_{L_x^{\infty}L_y^2}\\
    &\leq C\comi t^{\frac{1}{4}}\norm{\partial_y f }_{H^{8,0}_{\mu}}\norm{ f }_{H^{5,0}_{\mu_{\frac12}}}\leq C\varepsilon\comi t^{-\frac{1}{2}}\norm{\partial_y f }_{H^{8,0}_{\mu}}.
\end{align*}
Combining the above estimates yields that
\begin{equation}\label{ff}
\begin{aligned}
	&\sum^m_{k=1}\binom{m}{k} \norm{(\partial_x^k f )\partial_x^{m+1-k} f }_{L^2_{\mu}}   \leq C\varepsilon\comi t^{-\frac{1}{2}}\norm{\partial_y f }_{H^{8,0}_{\mu}}.
	\end{aligned}  
\end{equation}
  In the same way we have that  
\begin{equation}\label{uu}
	\sum^m_{k=1}\binom{m}{k} \norm{(\partial_x^ku)\partial_x^{m+1-k}u}_{L^2_{\mu}}   \leq C\varepsilon\comi t^{-\frac{1}{2}}\norm{\partial_yu}_{H^{8,0}_{\mu}},
\end{equation}
and that,  observing $\partial_y g=-\partial_x f $  and $\norm{ f }_{H^{8,0}_{\mu}}\leq (2\comi t)^{1\over 2} \norm{ \partial_yf }_{H^{8,0}_{\mu}}$ by Lemma \ref{lema},
\begin{equation}\label{kgyf}
\begin{aligned}
	&	\sum^{m-1}_{k=1}\binom{m-1}{k} \norm{(\partial_x^kg)\partial_x^{m-k}\partial_y f }_{L^2_{\mu}}  \leq C	\sum^{[\frac{m-1}{2}]}_{k=1}  \norm{ \mu_{\frac{1}{4}} \partial_x^kg}_{L^\infty}\norm{\mu_{\frac{1}{4}}\partial_x^{m-k}\partial_y f }_{L^2}\\
	&\qquad \qquad\qquad\qquad+C	\sum_{k=[\frac{m-1}{2}]+1}^{m-1}  \norm{ \mu_{\frac{1}{4}} \partial_x^kg}_{L_x^2 L_y ^\infty}\norm{\mu_{\frac{1}{4}}\partial_x^{m-k}\partial_y f }_{L_x^\infty L_y^2}\\
	&\leq C\comi t^{\frac{1}{4}} \norm{ f }_{H^{5,0}_{\mu_{\frac{3}{4}}}}\norm{\partial_y f }_{H^{8,0}_{\mu}}+ C\comi t^{\frac{1}{4}}\norm{ f }_{H^{8,0}_{\mu}}\norm{\partial_y f }_{H^{5,0}_{\mu_{\frac{3}{4}}}}\\
	&\leq C\Big(\comi t^{\frac{1}{4}} \norm{ f }_{H^{5,0}_{\mu_{\frac{3}{4}}}}+  \comi t^{\frac{1}{4}}\comi t^{\frac{1}{2}} \norm{\partial_y f }_{H^{5,0}_{\mu_{\frac{3}{4}}}}\Big)\norm{\partial_y f }_{H^{8,0}_{\mu}}\\
	&\leq  C\varepsilon\delta^{-\frac{1}{2}}\comi t^{-\frac{1}{2}}\norm{\partial_y f }_{H^{8,0}_{\mu}},
\end{aligned}
\end{equation}
the last inequality following from \eqref{asslow}. 
Similarly, 
\begin{equation*}
	\sum^{m-1}_{k=1}\binom{m-1}{k} \norm{(\partial_x^kv)\partial_x^{m-k}\partial_yu}_{L^2_{\mu}} \leq C\varepsilon\delta^{-\frac{1}{2}}\comi t^{-\frac{1}{2}}\norm{\partial_yu}_{H^{8,0}_{\mu}}.
\end{equation*}
Combining the two estimates above and  estimates \eqref{ff}-\eqref{uu}  with  \eqref{nsm1}, we conclude that
\begin{equation}\label{essm1}
	S_{m,1}\leq C\varepsilon\delta^{-\frac{1}{2}}\comi t^{-\frac{1}{2}}\norm{(\partial_yu, \partial_yf)}_{H^{8,0}_{\mu}}.
\end{equation}

  {\it Step 2 (Estimate on $S_{m,2}$)}. In view of \eqref{ublb}, we  follow the argument in \eqref{kgyf}  to compute, using 
  Lemmas \ref{lemv} and \ref{lemu8} and  assertion \eqref{asslow},  
\begin{align*}
    &\sum_{k=1}^{m-1}\binom{m-1}{k} \Big\|\frac{\partial_yu}{1+f}(\partial_x^k f )\partial_x^{m-k}v\Big\|_{L^2_{\mu}}\leq C\sum_{k=1}^{m-1}\norm{\mu_{\frac{1}{4}}\partial_yu}_{L^{\infty}} \norm{ (\partial_x^k f )\partial_x^{m-k}v}_{L^2_{\mu_{1\over2}}}\\
    &\leq   C\comi t^{\frac{1}{4}}\norm{\partial_y^2u}_{H^{1,0}_{\mu_{\frac{3}{4}}}}\Big(\comi t^{\frac{1}{4}}\norm{ f }_{H^{5,0}_{\mu_{\frac{3}{4}}}}\norm{u}_{H^{8,0}_{\mu}}+\comi t^{\frac{1}{4}}\norm{ f }_{H^{8,0}_{\mu}}\norm{u}_{H^{5,0}_{\mu_{\frac{3}{4}}}}\Big)\\
    &\leq C\varepsilon^2\comi t^{\frac{1}{2}}\norm{\partial_y^2\mathcal{U}}_{H^{1,0}_{\mu}}\comi t^{-\frac{5-\delta}{4}} \comi t^{-\frac{1-\delta}{4}}\leq C\varepsilon \comi t^{-\frac{1}{2}} \norm{\partial_y^2\mathcal{U}}_{H^{5,0}_{\mu}}.\end{align*}
    Similarly, 
    \begin{multline*}
    	\sum_{k=1}^{m-1}\binom{m-1}{k}\inner{\Big\|\frac{\partial_yu}{1+f}(\partial_x^kg)\partial_x^{m-k}u\Big\|_{L^2_{\mu}}+\Big\|\frac{\partial_yu}{1+f}(\partial_x^ku)\partial_x^{m-k}g\Big\|_{L^2_{\mu}}}\\+\sum_{k=1}^{m-1}\binom{m-1}{k}\Big\|\frac{\partial_yu}{1+f}(\partial_x^kv)\partial_x^{m-k}f\Big\|_{L^2_{\mu}}\leq C\varepsilon\comi t^{-\frac{1}{2}} \norm{\partial_y^2\mathcal{U}}_{H^{5,0}_{\mu}}.
    \end{multline*}
    Combining the estimates above with \eqref{nsm2}, we conclude that 
\begin{equation}\label{essm2}
    \norm{S_{m,2}}_{L^2_{\mu}}\leq  C\varepsilon \comi t^{-\frac{1}{2}}\norm{\partial_y^2\mathcal{U}}_{H^{5,0}_{\mu}}.
\end{equation}

 {\it Step 3 (Estimate on $S_{m,3}$)}. It just follows from straightforward computation. In fact, using  \eqref{ublb} again yields 
 \begin{align*}
 &\Big\|\Big[\frac{g\partial_y f }{1+f}+2\partial_y\Big(\frac{  \partial_yu}{1+f}\Big)  \Big]\partial_x^m f  
\Big\|_{L^2_{\mu}} \\
 &\leq C\norm{\mu_{\frac{1}{4}}\partial_x^m f }_{L_x^2L_y^{\infty}} \big\|\mu_{\frac{1}{4}}\big(g\partial_y f+\partial_y^2u+(\partial_yu) \partial_y f\big)\big\|_{L_x^{\infty}L_y^2} \\
 &\leq C\comi t^{1\over 4}\norm{ \partial_y  f }_{H^{8,0}_{\mu}} \big\|\mu_{\frac{1}{4}}\big(g\partial_y f+\partial_y^2u+(\partial_yu) \partial_y f\big)\big\|_{L_x^{\infty}L_y^2}.
 \end{align*}
 On the other hand,  using Lemmas \ref{lemu0} and \ref{lemv}  we compute 
 \begin{equation*}
 	\norm{\mu_{\frac{1}{4}} \partial_y^2u}_{L_x^{\infty}L_y^2}
\leq C \norm{ \partial_y^2\mathcal U}_{H^{5,0}_\mu}, 
 \end{equation*}
 and 
 \begin{align*}
 	\norm{\mu_{\frac{1}{4}} (\partial_yu) \partial_y f }_{L_x^{\infty}L_y^2} &\leq  \norm{\mu_{\frac{1}{8}}\partial_yu}_{L^{\infty}}\norm{\mu_{\frac{1}{8}}\partial_y f }_{L_x^{\infty}L_y^2}\leq C\comi t^{1\over 4} \norm{\partial_y^2u}_{H^{5,0}_{\mu_{1\over2}}}\norm{ \partial_y f }_{H^{5,0}_{\mu_{1\over2}}}\\
 	&\leq C\comi t^{1\over 4} \norm{\partial_y^2\mathcal U}_{H^{5,0}_{\mu}}\norm{ \partial_y f }_{H^{5,0}_{\mu_{1\over2}}}\leq C\varepsilon \delta^{-\frac{1}{2}}\norm{ \partial_y^2\mathcal U}_{H^{5,0}_\mu},
 \end{align*}
 the last inequality using \eqref{asslow}. 
 Similarly,  using Lemma  \ref{lema} additionally, 
 \begin{equation*}
 \begin{aligned}
 	&\norm{\mu_{\frac{1}{4}} g\partial_y f  }_{L_x^{\infty}L_y^2}   \leq C\comi t^{1\over 4} \norm{f}_{H^{5,0}_{\mu_{3\over4}}}\norm{ \partial_y f }_{H^{5,0}_{\mu_{1\over2}}}\\
 	&\leq C \comi t^{{1\over 4}+\frac{1}{2}} \norm{f}_{H^{5,0}_{\mu_{3\over4}}}\norm{ \partial_y^2 f }_{H^{5,0}_{\mu_{1\over2}}}   
  \leq C\varepsilon \norm{ \partial_y^2 \mathcal F}_{H^{5,0}_\mu}.
 \end{aligned}
 \end{equation*}
As a result, combining the above estimates, we conclude that
\begin{equation*}
	 \Big\|\Big[\frac{g\partial_y f }{1+f}+2\partial_y\Big(\frac{  \partial_yu}{1+f}\Big)  \Big]\partial_x^m f  
\Big\|_{L^2_{\mu}}\leq C\delta^{-\frac{1}{2}}\comi t^{1\over 4}\norm{ \partial_y  f }_{H^{8,0}_{\mu}} \norm{ (\partial_y^2 \mathcal U,  \partial_y^2 \mathcal F)}_{H^{5,0}_\mu}.
\end{equation*}
Following the above argument with slight modification, we can verify that 
the $L_\mu^2$-norms of the remaining terms in \eqref{nsm3}   are bounded above from above by  
\begin{equation*}
	C\delta^{-\frac{1}{2}}\comi t^{1\over 4}\inner{\norm{\partial_yu}_{H^{8,0}_{\mu}}+\norm{f}_{H^{8,0}_{\mu}}} \norm{ (\partial_y^2 \mathcal U,  \partial_y^2 \mathcal F)}_{H^{5,0}_\mu},
\end{equation*}
and thus
\begin{equation*}
\begin{aligned}
	S_{m,3}&\leq C\delta^{-\frac{1}{2}}\comi t^{1\over 4} \Big(\norm{\partial_yu}_{H^{8,0}_{\mu}}+\norm{ f }_{H^{8,0}_{\mu}} +\norm{ \partial_y  f }_{H^{8,0}_{\mu}} \Big)\norm{ (\partial_y^2 \mathcal U,  \partial_y^2 \mathcal F)}_{H^{5,0}_\mu}\\
	&\leq C\delta^{-\frac{1}{2}}\comi t^{3\over 4}  \norm{ (\partial_y u, \partial_y  f )}_{H^{8,0}_{\mu}}\norm{ (\partial_y^2 \mathcal U,  \partial_y^2 \mathcal F)}_{H^{5,0}_\mu}\leq C\delta^{-2}\comi t^{-{1\over 2}}\mathcal D_\delta(t),
	\end{aligned}
\end{equation*}
the second inequality using Lemma  \ref{lema} and the last inequality following from Lemma  \ref{lemu8} as well as definition \ref{ddelta} of $\mathcal D_\delta$. This with \eqref{essm1} and \eqref{essm2}  yields \eqref{upsmj}.   Thus the proof of Proposition \ref{propo2} is  completed.
\end{proof}

\section{Estimate on the auxiliary functions $\mathcal U$ and $\mathcal F$}\label{SecE2}
In this section, we first derive the estimate on $(\mathcal{U},\mathcal{F})$  defined in \eqref{eq3}, and then  complete the proof of Theorem \ref{thm:priori}.   

\begin{proposition}\label{propo3}
 Let $\varepsilon$ be the small constant given in Lemma \ref{lemu8}. Suppose  that the small assumption \eqref{sassm} holds.
 Then 
 \begin{equation}\label{eq4.1}
      \begin{aligned}
          \frac{d}{dt}\inner{\comi t^{\frac{5-\delta}{2}}\norm{(\mathcal{U},\mathcal{F})}^2_{H^{5,0}_{\mu}}}+\delta\comi t^{\frac{5-\delta}{2}}\norm{(\partial_y\mathcal{U},\partial_y\mathcal{F})}^2_{H^{5,0}_{\mu}}\leq C\varepsilon \delta^{-1} \mathcal{D}_\delta.
      \end{aligned}
  \end{equation}
\end{proposition}

\begin{proof} {\it Step 1).} 
Note that $\mathcal U$ and $\mathcal F$ satisfy the following equations (see Appendix \ref{appendix1} for the derivation):
 \begin{multline}\label{eq4.2}
        (\partial_t-\partial_y^2)\mathcal U+\frac{1}{\comi t} \mathcal{U}-\partial_x \mathcal{F}=-(u\partial_x+v\partial_y)u+\big(f\partial_x+g\partial_y\big) f  \\
        +\frac{y}{2\comi t}\int^{+\infty}_y\Big[(u\partial_x+v\partial_y)u - \big(f \partial_x+g\partial_y\big) f \Big]d\tilde y  
\end{multline}
and  
\begin{multline}\label{evf}
	(\partial_t-\partial_y^2)\mathcal F+\frac{1}{\comi t} \mathcal{F}-\partial_x \mathcal{U}=-(u\partial_x+v\partial_y)  f+\big(f\partial_x+g\partial_y\big)u \\
        +\frac{y}{2\comi t}\int^{+\infty}_y\Big[(u\partial_x+v\partial_y)  f  - \big(f\partial_x+g\partial_y\big)u\Big]d\tilde y.
\end{multline}
Observe 
the two equations above are supplemented with   the boundary condition that 
\begin{align}\label{eq4.3}
    \mathcal{U}|_{y=0}= \mathcal{F}|_{y=0}=0.
\end{align}
We first perform estimates for $\mathcal U$.   For any  $m\in\mathbb Z_+$ with $0\leq m\leq 5$, applying $\partial_x^m$ to the equation  \eqref{eq4.2} yields
\begin{equation}\label{umrm}
        (\partial_t-\partial_y^2)\partial_x^m\mathcal{U}+\frac{1}{\comi t}\partial_x^m\mathcal{U}  =\partial_x^{m+1}\mathcal{F}+R_{m,1}+R_{m,2}
      \end{equation} 
      with 
      \begin{equation}\label{rm1}
      \left\{
      \begin{aligned} 
  R_{m,1} =&    \partial_x^m\Big [-(u\partial_x+v\partial_y)u+\big(f \partial_x+g\partial_y\big) f \Big ],  \\ 
     R_{m,2}  = &\frac{y}{2\comi t}\int^{+\infty}_y\partial_x^m\Big[(u\partial_x+v\partial_y)u - \big(f\partial_x+g\partial_y\big) f \Big]d\tilde y.    \end{aligned}
     \right.
\end{equation}
We take the $L^2_{\mu}$-product with $\partial_x^m\mathcal{U}$  in \eqref{umrm} and then observe Remark \ref{rmkest}, to obtain that
\begin{equation}\label{eq4.5}
    \begin{aligned}
&\frac{d}{dt}\inner{\comi t^{\frac{5-\delta}{2}}\norm{\partial_x^m\mathcal{U}}^2_{L^2_{\mu}}}+\delta\comi t^{\frac{5-\delta}{2}}\norm{\partial_x^m\partial_y\mathcal{U}}^2_{L^2_{\mu}}\\
&\qquad   \leq  2 \comi t^{\frac{5-\delta}{2}}\inner{\partial_x^{m+1}\mathcal{F},\ \partial_x^m \mathcal{U}}_{L^2_{\mu}}+ 2\comi t^{\frac{5-\delta}{2}}\sum_{1\leq j\leq 2} \norm{R_{m,j}}_{L^2_{\mu}}\norm{\partial_x^m\mathcal{U}}_{L^2_{\mu}}.
    \end{aligned}
\end{equation}

{\it Step 2).} In this step  we will deal with the last term on the right-hand side of \eqref{eq4.5}, and   prove  that
\begin{equation}\label{rm}
\sum_{1\leq j\leq 2}\norm{R_{m,j}}_{L^2_{\mu}}\leq C\varepsilon\comi t^{-\frac{1}{2}}\norm{(\partial_y\mathcal{U}, \partial_y\mathcal{F})}_{H^{5,0}_{\mu}}.
\end{equation}  
For each $0\leq m \leq 5$ we have  $  \norm{\partial_x^m(u\partial_xu)}_{L^2_{\mu}}
   \leq \norm{  (\mu_{1\over4}u) \mu_{1\over4}\partial_xu }_{H_x^{5} L_y^2}$ and thus
\begin{equation}\label{uxu}
    \norm{\partial_x^m(u\partial_xu)}_{L^2_{\mu}}
   \leq C \norm{   \mu_{1\over4}u  }_{H_x^{5}L_y^\infty}\norm{   \mu_{1\over4}\partial_xu }_{H_x^{5}L_y^2}  
    \leq C\comi t^{\frac{1}{4}} \norm{\partial_yu}_{H^{5,0}_{\mu_{3/4}}}\norm{\partial_xu}_{H^{5,0}_{\mu_{3/4}}},
\end{equation}
the last inequality using Lemma \ref{lemv}. For the last term on the right-hand side, it follows from the interpolation inequality   that
\begin{equation}\label{eq4.6}
\begin{aligned}
   \norm{\partial_xu}_{H^{5,0}_{\mu_{3/4}}}&\leq  \norm{u}_{H^{6,0}_{\mu_{3/4}}} \leq C\norm{u}_{H^{8,0}_{\mu_{3/4}}}^{\frac{1}{3}}\norm{u}_{H^{5,0}_{\mu_{3/4}}}^{\frac{2}{3}} \\
   &\leq C\varepsilon\comi t^{-\frac{1-\delta}{4}\times\frac{1}{3}}\comi t^{-\frac{5-\delta}{4}\times\frac{2}{3}}\leq C\varepsilon\comi t^{-\frac{11}{12}+\frac{\delta}{4}}, 
\end{aligned}
\end{equation}
where    in the last line   we  use  Lemma  \ref{lemu8} and  assertion \eqref{asslow}.    
As a result,  we substitute \eqref{eq4.6} into \eqref{uxu} and then   use Lemma \ref{lemu0} as well as the fact that $0<\delta\leq \frac{1}{25}$, to conclude  that
\begin{equation}
\begin{aligned}\label{eq4.7}
    \norm{\partial_x^m(u\partial_xu)}_{L^2_{\mu}} \leq   C\varepsilon\comi t^{-\frac{11}{12}+\frac{\delta}{4}} \comi t^{1\over 4}\norm{\partial_yu}_{H^{5,0}_{\mu_{3/4}}}\leq   C\varepsilon\comi t^{-\frac{1}{2}}  \norm{\partial_y\mathcal U}_{H^{5,0}_{\mu}} .
\end{aligned}
\end{equation}
Following the above argument with slight modification, we can verify that,   
for any $0\leq m \leq 5$,
\begin{multline*}
    \norm{\partial_x^m(v\partial_yu)}_{L^2_{\mu}}\leq C\norm{\mu_{1/4} v}_{H_x^{5}L_y^\infty}  \norm{\mu_{1/4}\partial_yu}_{H_x^{5}L_y^2} \\
     \leq C\comi t^{1\over 4} \norm{\partial_xu}_{H^{5,0}_{\mu_{3/4}}}  \norm{ \partial_yu}_{H^{5,0}_{\mu_{3/4}}}
    \leq C \varepsilon \comi t^{-{1\over 2}}   \norm{ \partial_y\mathcal U}_{H^{5,0}_{\mu}},
\end{multline*}
where  the first inequality in   the last line  follows from Lemma \ref{lemv}  and   the last inequality holds because of  \eqref{eq4.6}. 
Similarly,
\begin{equation*}
	 \norm{\partial_x^m(f\partial_xf)}_{L^2_{\mu}} +\norm{\partial_x^m(g\partial_yf)}_{L^2_{\mu}}
    \leq C \varepsilon \comi t^{-{1\over 2}}   \norm{ \partial_y\mathcal F}_{H^{5,0}_{\mu}}.
\end{equation*} 
 Combining the above estimates and \eqref{eq4.7} with \eqref{rm1}   yields
\begin{align}\label{eq4.9}
    \norm{R_{m,1}}_{L^2_{\mu}} \leq C\varepsilon\comi t^{-\frac{1}{2}}\norm{(\partial_y\mathcal{U}, \partial_y\mathcal{F})}_{H^{5,0}_{\mu}}.
\end{align}
In view of \eqref{rm1},  we apply  \eqref{eq2.6} in Lemma \ref{lema} with 
\begin{equation*}
	h=\int^{+\infty}_y\partial_x^m\Big[(u\partial_x+v\partial_y)u - \big(f \partial_x+g\partial_y\big) f \Big]d\tilde y, 
\end{equation*}
to obtain that, observing $\partial_yh=R_{m,1}, $  
\begin{align*}
    \norm{R_{m,2}}_{L^2_{\mu}}= \Big\| \frac{y}{2\comi t} h\Big\|_{L^2_{\mu}}  \leq C   \norm{R_{m,1}}_{L^2_{\mu}},
\end{align*}
 which with \eqref{eq4.9} yields \eqref{rm}. 

{\it Step 3).}  
For each  $m\leq 5,$ using  \eqref{rm} and Lemma \ref{lema}   yields  
\begin{align*}
 \comi t^{\frac{5-\delta}{2}}\sum_{1\leq j\leq 2} \norm{R_{m,j}}_{L^2_{\mu}}\norm{\partial_x^m\mathcal{U}}_{L^2_{\mu}}
 &\leq C\varepsilon  \comi t^{\frac{5-\delta}{2}} \comi t^{-\frac{1}{2}} \norm{(\partial_y\mathcal{U}, \partial_y\mathcal{F})}_{H^{5,0}_{\mu}}
 \norm{ \mathcal{U}} _{H^{5,0}_{\mu}}\\
 &\leq C\varepsilon  \comi t^{\frac{5-\delta}{2}}  \norm{(\partial_y\mathcal{U}, \partial_y\mathcal{F})}_{H^{5,0}_{\mu}}^2\leq C\eps\delta^{-1} \mathcal D_{\delta},
 \end{align*}
the last inequality using definition \eqref{ddelta} of $ \mathcal D_{\delta}$.  Substituting the above estimate into \eqref{eq4.5},  we conclude that
\begin{equation*}
	    \begin{aligned}
&\frac{d}{dt}\inner{\comi t^{\frac{5-\delta}{2}}\norm{\partial_x^m\mathcal{U}}^2_{L^2_{\mu}}}+\delta\comi t^{\frac{5-\delta}{2}}\norm{\partial_x^m\partial_y\mathcal{U}}^2_{L^2_{\mu}}\\
&\qquad \qquad \leq  2 \comi t^{\frac{5-\delta}{2}}\inner{\partial_x^{m+1}\mathcal{F},\ \partial_x^m \mathcal{U}}_{L^2_{\mu}}+C\eps\delta^{-1} \mathcal D_{\delta}.
    \end{aligned}
\end{equation*}
 Similarly,  we repeat the above argument for the evolution equation \eqref{evf} of $\mathcal F$ to conclude that 
\begin{equation*}
	    \begin{aligned}
&\frac{d}{dt}\Big(\comi t^{\frac{5-\delta}{2}}\norm{\partial_x^m\mathcal{F}}^2_{L^2_{\mu}}\Big)+\delta\comi t^{\frac{5-\delta}{2}}\norm{\partial_x^m\partial_y\mathcal{F}}^2_{L^2_{\mu}}  \\
&\qquad\qquad \leq  2 \comi t^{\frac{5-\delta}{2}}\inner{\partial_x^{m+1}\mathcal{U},  \partial_x^m \mathcal{F}}_{L^2_{\mu}}+ C\varepsilon \delta^{-1} \mathcal D_\delta.
    \end{aligned}
\end{equation*}
Combining the two estimates above and observing the fact that
\begin{equation*}
	\inner{\partial_x^{m+1}\mathcal{F},\ \partial_x^m \mathcal{U}}_{L^2_{\mu}}+\inner{\partial_x^{m+1}\mathcal{U},  \partial_x^m \mathcal{F}}_{L^2_{\mu}}=0,
\end{equation*}
we obtain   the   assertion  in Proposition \ref{propo3}.   The proof of   is thus completed.
\end{proof}

\begin{proposition}\label{propo4} Under the hypothesis of Proposition \ref{propo3}, we have
  \begin{multline}\label{eq4.14}
	\frac{d}{dt}\inner{\frac{\delta}{2}  \comi t^{\frac{7-\delta}{2}}\norm{ (\partial_y\mathcal{U}, \partial_y\mathcal{F})}^2_{H^{5,0}_{\mu}}}+ \frac{\delta^2}{2}\comi t^{\frac{7-\delta}{2}}\norm{(\partial_y^2\mathcal{U}, \partial_y^2\mathcal{F})}^2_{H^{5,0}_{\mu}}\\
	\leq \frac{\delta}{2} \comi t^{\frac{5-\delta}{2}}\norm{ (\partial_y\mathcal{U},\partial_y\mathcal{F})}^2_{H^{5,0}_{\mu}}+C\varepsilon  \delta^{-1} \mathcal D_\delta.
\end{multline}
\end{proposition}

\begin{proof}
For any $0\leq m\leq 5$, applying $\partial_y$ to   equation  \eqref{umrm} of $\partial_x^m \mathcal{U}$ yields
\begin{equation}\label{yufr}
(\partial_t-\partial_y^2)\partial_x^m\partial_y \mathcal{U} +\frac{1}{\comi t} \partial_x^m\partial_y \mathcal{U}  = \partial_x^{m+1} \partial_y\mathcal{F}+\partial_y R_{m,1}  + \partial_y R_{m,2}, 
\end{equation}
recalling $R_{m,1}$ and $R_{m,2}$ are given in \eqref{rm1}. 
By the boundary condition in \eqref{eq2}, the trace at $y=0$ of the source term in \eqref{eq4.2} is equal to $0$. This with  \eqref{eq4.3} yields that, taking  $y=0$ in equation  \eqref{eq4.2},  
\begin{equation*}
	\partial_y^2\mathcal U|_{y=0}=0\  \textrm{ and }\ R_{m,1}|_{y=0} =R_{m,2}|_{y=0}=0.
\end{equation*} 
 Thus we take the $L^2_{\mu}$-product with $\partial_x^m\partial_y\mathcal{U}$ on both sides of \eqref{yufr} and then apply Remark \ref{rmkest} to get
 \begin{equation*}     \begin{aligned}
       & \frac{d}{dt}\norm{\partial_x^m\partial_y\mathcal{U}}^2_{L^2_{\mu}}+\frac{5-\delta}{2\comi t}\norm{\partial_x^m\partial_y\mathcal{U}}^2_{L^2_{\mu}}+ \delta\norm{\partial_x^m\partial_y^2\mathcal{U}}^2_{L^2_{\mu}}\\
        &\leq  2\inner{\partial_x^{m+1}\partial_y\mathcal{F},\ \partial_x^m\partial_y \mathcal{U}}_{L^2_{\mu}} +2\sum_{1\leq j\leq 2}\inner{\partial_y R_{m,j},\  \partial_x^m\partial_y \mathcal{U}}_{L^2_{\mu}},
    \end{aligned}
\end{equation*}
or equivalently, 
\begin{equation*} 
    \begin{aligned}
       & \frac{d}{dt}\norm{\partial_x^m\partial_y\mathcal{U}}^2_{L^2_{\mu}}+\frac{7-\delta}{2\comi t}\norm{\partial_x^m\partial_y\mathcal{U}}^2_{L^2_{\mu}}+\delta\norm{\partial_x^m\partial_y^2\mathcal{U}}^2_{L^2_{\mu}}\\
        &\leq \frac{1}{\comi t}\norm{\partial_x^m\partial_y\mathcal{U}}^2_{L^2_{\mu}}+ 2\inner{\partial_x^{m+1}\partial_y\mathcal{F},\ \partial_x^m\partial_y \mathcal{U}}_{L^2_{\mu}} +2\sum_{1\leq j\leq 2}\norm{R_{m,j}}_{L^2_{\mu}}\norm{\partial_x^m\partial_y^2\mathcal{U}}_{L^2_{\mu}}.
    \end{aligned}
\end{equation*}
Multiplying the above equation  by $\comi t^{\frac{7-\delta}{2}}$ and observing 
\begin{equation*}
	\frac{d}{dt}\comi t^{\frac{7-\delta}{2}}= \comi t^{\frac{7-\delta}{2}}\frac{7-\delta}{2\comi t},
\end{equation*}
we get
\begin{equation*} 
    \begin{aligned}
&\frac{d}{dt}\inner{\comi t^{\frac{7-\delta}{2}}\norm{\partial_x^m\partial_y\mathcal{U}}^2_{L^2_{\mu}}}+ \delta\comi t^{\frac{7-\delta}{2}}\norm{\partial_x^m\partial_y^2\mathcal{U}}^2_{L^2_{\mu}}\leq \comi t^{\frac{5-\delta}{2}}\norm{\partial_x^m\partial_y\mathcal{U}}^2_{L^2_{\mu}}\\
&\qquad  + 2\comi t^{\frac{7-\delta}{2}}\inner{\partial_x^{m+1}\partial_y\mathcal{F},\ \partial_x^m \partial_y\mathcal{U}}_{L^2_{\mu}}+  2\comi t^{\frac{7-\delta}{2}} \sum_{1\leq j\leq 2}\norm{R_{m,j}}_{L^2_{\mu}}\norm{\partial_x^m\partial_y^2\mathcal{U}}_{L^2_{\mu}}.
    \end{aligned}
\end{equation*}
For the last term on the right-hand side,  we use   \eqref{rm} and Lemma \ref{lema} as well as definition \eqref{ddelta} of $\mathcal D_\delta,$ to conclude that, for each $m\leq 5,$    
\begin{align*}
	\comi t^{\frac{7-\delta}{2}} \sum_{1\leq j\leq 2}\norm{R_{m,j}}_{L^2_{\mu}}\norm{\partial_x^m\partial_y^2\mathcal{U}}_{L^2_{\mu}}&\leq C\varepsilon \comi t^{\frac{7-\delta}{2}}    \comi t^{-\frac{1}{2}}\norm{(\partial_y\mathcal{U},\partial_y\mathcal{F})}_{H^{5,0}_{\mu}}  \norm{\partial_y^2\mathcal{U}}_{H^{5,0}_{\mu}}\\
	&\leq C\varepsilon\comi t^{\frac{7-\delta}{2}} \norm{(\partial_y^2\mathcal{U}, \partial_y^2\mathcal{F})}_{H^{5,0}_{\mu}}^2 \leq C\varepsilon \delta^{-2}\mathcal D_\delta.
\end{align*}
Combining the two  estimates above, we have, for any $0\leq m\leq 5,$
\begin{multline*}
	\frac{d}{dt}\inner{\comi t^{\frac{7-\delta}{2}}\norm{\partial_x^m\partial_y\mathcal{U}}^2_{L^2_{\mu}}}+ \delta\comi t^{\frac{7-\delta}{2}}\norm{\partial_x^m\partial_y^2\mathcal{U}}^2_{L^2_{\mu}}\\
	\leq \comi t^{\frac{5-\delta}{2}}\norm{\partial_x^m\partial_y\mathcal{U}}^2_{L^2_{\mu}}
+  2\comi t^{\frac{7-\delta}{2}}\inner{\partial_x^{m+1}\partial_y\mathcal{F},\ \partial_x^m \partial_y\mathcal{U}}_{L^2_{\mu}}+C\varepsilon  \delta^{-2} \mathcal D_\delta.
\end{multline*}
Similarly,  for any $0\leq m\leq 5$
it holds that
\begin{multline*}
	\frac{d}{dt}\inner{\comi t^{\frac{7-\delta}{2}}\norm{\partial_x^m\partial_y\mathcal{F}}^2_{L^2_{\mu}}}+ \delta\comi t^{\frac{7-\delta}{2}}\norm{\partial_x^m\partial_y^2\mathcal{F}}^2_{L^2_{\mu}}\\
	\leq \comi t^{\frac{5-\delta}{2}}\norm{\partial_x^m\partial_y\mathcal{F}}^2_{L^2_{\mu}}
+  2\comi t^{\frac{7-\delta}{2}}\inner{\partial_x^{m+1}\partial_y\mathcal{U},\ \partial_x^m \partial_y\mathcal{F}}_{L^2_{\mu}}+C\varepsilon  \delta^{- 2} \mathcal D_\delta.
\end{multline*}
As a result, we   combine the above two estimates  with the fact that
\begin{equation*}
	\inner{\partial_x^{m+1}\partial_y\mathcal{F},\ \partial_x^m \partial_y\mathcal{U}}_{L^2_{\mu}}+\inner{\partial_x^{m+1}\partial_y\mathcal{U},\ \partial_x^m \partial_y\mathcal{F}}_{L^2_{\mu}}=0,
\end{equation*}
to obtain that
\begin{multline*}
	\frac{d}{dt}\inner{\comi t^{\frac{7-\delta}{2}}\norm{ (\partial_y\mathcal{U}, \partial_y\mathcal{F})}^2_{H^{5,0}_{\mu}}}+ \delta\comi t^{\frac{7-\delta}{2}}\norm{(\partial_y^2\mathcal{U}, \partial_y^2\mathcal{F})}^2_{H^{5,0}_{\mu}}\\
	\leq \comi t^{\frac{5-\delta}{2}}\norm{ (\partial_y\mathcal{U},\partial_y\mathcal{F})}^2_{H^{5,0}_{\mu}}+C\varepsilon  \delta^{-2} \mathcal D_\delta.
\end{multline*}
Multiplying the above estimate by $\frac{\delta}{2}$ yields the assertion in Proposition \ref{propo4}. The proof is thus completed.
\end{proof}

\begin{proof}[Completing the proof  of Theorem \ref{thm:priori}]
We take summation of  estimates \eqref{eq4.1} and \eqref{eq4.14} and observe  the first term on the right-hand side of   \eqref{eq4.14} can be absorbed by that on  the left side of \eqref{eq4.1}; this gives 
\begin{multline*}
	\frac{d}{dt}	 \sum_{0\leq k\leq 1}\Big[\Big(\frac{\delta}{2} \Big)^{k} \comi t^{\frac{5-\delta}{2}+k} \norm{(\partial_y^k\mathcal{U},\partial_y^k\mathcal{F})}^2_{H^{5,0}_{\mu}}\Big]\\
	+\sum_{0\leq k\leq 1}\Big(\frac{\delta}{2} \Big)^{k+1} \comi t^{\frac{5-\delta}{2}+k} \norm{(\partial_y^{k+1}\mathcal{U},\partial_y^{k+1}\mathcal{F})}^2_{H^{5,0}_{\mu}}\leq C\varepsilon  \delta^{-1} \mathcal D_\delta. 
	\end{multline*}
	Moreover, by Proposition \ref{propo2},
	\begin{align*}
		 \frac{d}{dt}   \sum_{ m\leq 8}   \comi t^{\frac{1-\delta}{2}}\norm{(u_m,f_m)}^2_{L^2_{\mu}} +\sum_{ m\leq 8}  \delta\comi t^{\frac{1-\delta}{2}}\norm{(\partial_yu_m,\partial_yf_m)}^2_{L^2_{\mu}} \leq C\varepsilon\delta^{-2} \mathcal{D}_\delta.
	\end{align*}
In view of definitions \eqref{edelta} and \eqref{ddelta} of $\mathcal{E}_\delta$ and $\mathcal{D}_\delta$, we combine the above estimates  to conclude that 
\begin{align*}
 \forall \ t>0,\quad   \frac{d}{dt}\mathcal{E}_\delta(t)+ \mathcal{D}_\delta(t)\leq C\varepsilon \delta^{-2}\mathcal{D}_\delta(t).
\end{align*}
We choose $\varepsilon$ small enough such that $ C\varepsilon \delta^{-2}\leq \frac 12$. Then 
 \begin{equation*}
\forall \ t>0,\quad\mathcal E_\delta (t)+\frac{1}{2}\int^{t}_{0}\mathcal D_\delta (s)ds \leq \mathcal{E}_{\delta}(0)\leq \varepsilon^2.
\end{equation*}
  The proof of Theorem \ref{thm:priori} is thus completed.
\end{proof}

 \appendix
\section{Equations for $\mathcal{U}$ and $\mathcal{F}$}\label{appendix1} 
 This part is devoted to  deriving the  equations solved by $\mathcal{U}$ and $\mathcal F$, recalling 
\begin{equation}\label{uf}
	\mathcal U=u- \frac{y}{2\comi t}\int^{+\infty}_yud\tilde y \ \textrm{ and }\ \mathcal F=f- \frac{y}{2\comi t}\int^{+\infty}_yfd\tilde y.
\end{equation}  
To do so,  
integrating on $[y,+\infty)$ with respect to $y$ for the first equation in \eqref{eq2}, we get
 \begin{align*}
    \inner{\partial_t-\partial_y^2}\int^{+\infty}_yud\tilde y-\partial_x\int^{+\infty}_y f d\tilde y=\int^{+\infty}_y\Big[\big(f \partial_x+g\partial_y\big) f  - (u\partial_x+v\partial_y)u \Big]d\tilde y.
 \end{align*}
Multiplying the above equation by $\frac{y}{2\comi t}$ and observing that
 \begin{align*}
    & \frac{y}{2\comi t}\inner{\partial_t-\partial_y^2}\int^{+\infty}_yud\tilde y\\
    &\qquad\qquad =\inner{\partial_t-\partial_y^2}\inner{\frac{y}{2\comi t}\int^{+\infty}_yud\tilde y}+
  \underbrace{  \frac{y}{2\comi t^2}\int^{+\infty}_yud\tilde y-\frac{1}{\comi t}u}_{=-\frac{1}{\comi t}\mathcal U},
 \end{align*}
we get 
 \begin{multline*}
     \inner{\partial_t-\partial_y^2}\inner{\frac{y}{2\comi t}\int^{+\infty}_yud\tilde y}-\partial_x\inner{\frac{y}{2\comi t}\int^{+\infty}_y f d\tilde y}-\frac{1}{\comi t}\mathcal{U}\\=\frac{y}{2\comi t}\int^{+\infty}_y\Big[\big(f\partial_x+g\partial_y\big) f  - (u\partial_x+v\partial_y)u\Big] d\tilde y.
 \end{multline*}
  Then in view of \eqref{uf},  we subtract the above equation by  the first equation in \eqref{eq2}   to get
 \begin{multline*}
       (\partial_t-\partial_y^2)\mathcal U+\frac{1}{\comi t} \mathcal{U}-\partial_x \mathcal{F}=-(u\partial_x+v\partial_y)u+\big(f\partial_x+g\partial_y\big) f \\
        +\frac{y}{2\comi t}\int^{+\infty}_y\Big[(u\partial_x+v\partial_y)u - \big(f \partial_x+g\partial_y\big) f \Big]d\tilde y.
\end{multline*}
In the same way, we have 
\begin{multline*}
	 (\partial_t-\partial_y^2)\mathcal F+\frac{1}{\comi t} \mathcal{F}-\partial_x \mathcal{U}=-(u\partial_x+v\partial_y)  f+(f\partial_x+g\partial_y)u\\
        +\frac{y}{2\comi t}\int^{+\infty}_y\Big[(u\partial_x+v\partial_y) f  - (f\partial_x+g\partial_y)u\Big]d\tilde y.
\end{multline*}
 
 \section{Equations for $u_m$ and $f_m$}\label{appumfm}
 
 We begin with the evolution  equation solved by $g$. In view of \eqref{zeromean}, we may write
 \begin{equation*}
 	g=-\int_0^y \partial_x f d\tilde y =\int^{\infty}_{y}\partial_x f  d\tilde y\ \textrm{ and }\ v=-\int_0^y \partial_x u d\tilde y =\int^{\infty}_{y}\partial_x u  d\tilde y.
 \end{equation*}
 Using integration by part yields  
 \begin{align*}
 	 \int_y^{+\infty} v\partial_yfd\tilde y=\int_y^{+\infty} (\partial_xu )fd\tilde y-vf,  \  \int_y^{+\infty} g\partial_y ud\tilde y=\int_y^{+\infty} (\partial_xf )ud\tilde y-gu.
 \end{align*}
 Then integrating   the second equation in \eqref{eq2} over $(y, +\infty)$ with respect to the normal variable, we have
\begin{align*}
    \partial_t\int^{\infty}_{y} f (\tilde y)d\tilde y-v f -\partial_y^2\int^{\infty}_{y} f (\tilde y)d\tilde y=\partial_x\int^{\infty}_{y}u(\tilde y)d\tilde y-gu.
\end{align*}
Applying $\partial_x$ to the above equation yields
\begin{align*}
       \partial_tg-(\partial_xv) f -v\partial_x f -\partial_y^2g=\partial_xv-(\partial_xg)u-g\partial_xu. 
\end{align*}
This with the divergence-free condition  yields
\begin{align}\label{dix1}
        \partial_tg+u\partial_xg+v\partial_yg-\partial_y^2g=(1+f)\partial_xv-g\partial_xu.    
\end{align} 
Applying $\partial_x^{m-1}$ to equation \eqref{dix1}   and using divergence conditions again, we have 
\begin{equation}\label{m-1g}
	\begin{aligned}
	&\inner{\partial_t+u\partial_x+v\partial_y-\partial_y^2}\partial_x^{m-1}g= (1+f)\partial_x^mv-g\partial_x^mu+I_m, 
	\end{aligned}
\end{equation}
where
\begin{multline*}
	I_m=\sum^{m-1}_{k=1}\binom{m-1}{k}\left[(\partial_x^k f )\partial_x^{m-k}v-(\partial_x^kg)\partial_x^{m-k}u \right]\\
	-\sum^{m-1}_{k=1}\binom{m-1}{k}\left[ (\partial_x^ku)\partial_x^{m-k}g-(\partial_x^kv)\partial_x^{m-k} f \right].
\end{multline*}
On the other hand, we apply $\partial_x^m$ to first second equations in \eqref{eq2} to obtain
\begin{multline}\label{xmu}
	  \inner{\partial_t+u\partial_x+v\partial_y-\partial_y^2}\partial_x^mu-\inner{(1+f)\partial_x+g\partial_y}\partial_x^m f \\ =-(\partial_x^mv)\partial_yu+(\partial_x^mg)\partial_y f +J_m 
	  	\end{multline} 
with
\begin{multline*}
	J_m=\sum^m_{k=1}\binom{m}{k}\left[(\partial_x^k f )\partial_x^{m-k+1} f -(\partial_x^ku)\partial_x^{m-k+1}u\right]\\
  +\sum^{m-1}_{k=1}\binom{m}{k}\left[(\partial_x^kg)\partial_x^{m-k}\partial_y f -(\partial_x^kv)\partial_x^{m-k}\partial_yu\right].
\end{multline*}
Multiplying  equation   \eqref{m-1g} by $\frac{\partial_yu}{1+f}$ and taking summation with the  equation \eqref{xmu} for $\partial_x^m u$, we obtain that, recalling $u_m=\partial_x^mu+\frac{\partial_yu}{1+f}\partial_x^{m-1}g, $
\begin{equation}\label{umxmf}
\begin{aligned}
   &\inner{\partial_t+u\partial_x+v\partial_y-\partial_y^2}u_m-\inner{(1+f)\partial_x+g\partial_y}\partial_x^m f  \\
   &\qquad =(\partial_x^mg)\partial_y f +J_m-\frac{g\partial_yu}{1+f}\partial_x^mu+\frac{\partial_yu}{1+f} I_m+K_m, 
   \end{aligned}
\end{equation}
where, denoting by $[P,Q]$ the commutator between two operators $P$ and $Q,$
 \begin{equation*}
 	\begin{aligned}
 	K_m&=\Big[\inner{\partial_t+u\partial_x+v\partial_y-\partial_y^2},\  \frac{\partial_yu}{1+f}\Big ]\partial_x^{m-1}g\\
 	&=\left[\partial_x\partial_y f +\frac{g\partial_y^2 f }{1+f} \right]\partial_x^{m-1}g-\left[ \frac{(\partial_yu)\partial_xu}{1+f}+\frac{g(\partial_yu)^2}{(1+f)^2}\right]\partial_x^{m-1}g\\
 	&\quad+\left[2\frac{(\partial_y f )\partial_y^2u}{(1+f)^2}-2\frac{(\partial_y f )^2\partial_yu}{(1+f)^3}\right]\partial_x^{m-1}g+2\inner{\partial_y\frac{\partial_yu}{1+f}}\partial_x^m f. 	
 	\end{aligned}
 \end{equation*}
 Moreover,  recalling $f_m=\partial_x^mf+\frac{\partial_yf}{1+f}\partial_x^{m-1}g, $
 \begin{align*}
    &-\big ((1+f)\partial_x+g\partial_y\big )\partial_x^m f = -\big ((1+f)\partial_x+g\partial_y\big )\Big(f_m-\frac{\partial_y f }{1+f}\partial_x^{m-1}g\Big)\\
    &\qquad= -\big ((1+f)\partial_x+g\partial_y\big )f_m+(\partial_y f )\partial_x^mg-\frac{g\partial_y f }{1+f}\partial_x^m f \\
    &\qquad\quad+\left[\partial_x\partial_y f +\frac{g\partial_y^2 f }{1+f} \right]\partial_x^{m-1}g-\left[ \frac{(\partial_y f )\partial_x f }{1+f}+\frac{g(\partial_y f )^2}{(1+f)^2}\right]\partial_x^{m-1}g.
\end{align*}
Substituting the above equation into \eqref{umxmf} and observing the term $(\partial_x^mg)\partial_y f$ therein is cancelled out, we have
\begin{equation*}
	\begin{aligned}
	&\inner{\partial_t+u\partial_x+v\partial_y-\partial_y^2}u_m-\inner{(1+f)\partial_x+g\partial_y}f_m\\
	&\quad=\frac{g\partial_y f }{1+f}\partial_x^m f -\frac{g\partial_yu}{1+f}\partial_x^mu+J_m+\frac{\partial_yu}{1+f} I_m	+\widetilde K_m,
	\end{aligned}
\end{equation*}
 where
 \begin{equation*}
 \begin{aligned}
 	\widetilde K_m=&\left[ \frac{(\partial_y f )\partial_x f }{1+f}+\frac{g(\partial_y f )^2}{(1+f) ^2}\right]\partial_x^{m-1}g-\left[ \frac{(\partial_yu)\partial_xu}{1+f}+\frac{g(\partial_yu)^2}{(1+f)^2}\right]\partial_x^{m-1}g\\
 &+	\left[2\frac{(\partial_y f )\partial_y^2u}{(1+f)^2}-2\frac{(\partial_y f )^2\partial_yu}{(1+f)^3}\right]\partial_x^{m-1}g+2\inner{\partial_y\frac{\partial_yu}{1+f}}\partial_x^m f .
 	\end{aligned}
 \end{equation*}
 Equivalently, 
 \begin{equation*}  	\inner{\partial_t+u\partial_x+v\partial_y-\partial_y^2}u_m-\inner{(1+f)\partial_x+g\partial_y}f_m=\sum_{1\leq j\leq 3} S_{m,j}
 \end{equation*}
 with
 \begin{equation}\label{sm1}
 \begin{aligned}
 	S_{m,1}=J_m=& \sum^m_{k=1}\binom{m}{k}\left[(\partial_x^k f )\partial_x^{m-k+1} f -(\partial_x^ku)\partial_x^{m-k+1}u\right]\\
  &+\sum^{m-1}_{k=1}\binom{m}{k}\left[(\partial_x^kg)\partial_x^{m-k}\partial_y f -(\partial_x^kv)\partial_x^{m-k}\partial_yu\right],
 \end{aligned}
 \end{equation}
 and
 \begin{equation}\label{sm2}
 \begin{aligned}
 	S_{m,2}=\frac{\partial_yu}{1+f} I_m	=&\frac{\partial_yu}{1+f} \sum^{m-1}_{k=1}\binom{m-1}{k}\left[(\partial_x^k f )\partial_x^{m-k}v-(\partial_x^kg)\partial_x^{m-k}u \right]\\
	&-\frac{\partial_yu}{1+f}\sum^{m-1}_{k=1}\binom{m-1}{k} \left[ (\partial_x^ku)\partial_x^{m-k}g-(\partial_x^kv)\partial_x^{m-k} f \right],
 \end{aligned}
 \end{equation}
 and
  \begin{equation}\label{sm3}
 \begin{aligned}
 	S_{m,3}&=\frac{g\partial_y f }{1+f}\partial_x^m f -\frac{g\partial_yu}{1+f}\partial_x^mu+\widetilde K_m\\
	&= \left[\frac{g\partial_y f }{1+f}+2\partial_y\inner{\frac{\partial_yu}{1+f}}\right]\partial_x^m f -\frac{g\partial_yu}{1+f}\partial_x^mu-2\frac{(\partial_y f )^2\partial_yu}{(1+f)^3} \partial_x^{m-1}g\\
	&\quad+\left[ \frac{(\partial_y f )\partial_x f -(\partial_yu)\partial_xu}{1+f}+\frac{g(\partial_y f )^2-g(\partial_yu)^2+2(\partial_y f )\partial_y^2u}{(1+f) ^2}\right]\partial_x^{m-1}g.
 \end{aligned}
 \end{equation}
By similar computation to above, we have that
  \begin{equation*} 
 	\inner{\partial_t+u\partial_x+v\partial_y-\partial_y^2}f_m-\inner{(1+f)\partial_x+g\partial_y}u_m=\sum_{1\leq j\leq 3} \tilde S_{m,j}
 \end{equation*}
 with 
  \begin{equation}\label{sm1t}
 \begin{aligned}
 	\tilde S_{m,1}= & \sum^m_{k=1}\binom{m}{k}\left[(\partial_x^k f )\partial_x^{m-k+1}u-(\partial_x^ku)\partial_x^{m-k+1}f\right]\\
  &+\sum^{m-1}_{k=1}\binom{m}{k}\left[(\partial_x^kg)\partial_x^{m-k}\partial_yu-(\partial_x^kv)\partial_x^{m-k}\partial_yf\right],
 \end{aligned}
 \end{equation}
 \begin{equation}\label{sm2t}
 \begin{aligned}
 	\tilde S_{m,2}= & \sum^{m-1}_{k=1}\binom{m-1}{k}\frac{\partial_yf}{1+f}\left[(\partial_x^k f )\partial_x^{m-k}v-(\partial_x^kg)\partial_x^{m-k}u \right]\\
	&-\sum^{m-1}_{k=1}\binom{m-1}{k}\frac{\partial_yf}{1+f}\left[ (\partial_x^ku)\partial_x^{m-k}g-(\partial_x^kv)\partial_x^{m-k} f \right],
 \end{aligned}
 \end{equation}
 and
  \begin{equation}\label{sm3t}
 \begin{aligned}
 	\tilde S_{m,3}& = \left[\frac{g\partial_yu}{1+f}+2\partial_y\inner{\frac{\partial_yf}{1+f}}\right]\partial_x^m f -\frac{g\partial_yf}{1+f}\partial_x^mu\\
	&\quad+\left[ \frac{ (\partial_xu)\partial_y f -(\partial_yu)\partial_xf}{1+f}+\frac{ 2(\partial_y f )\partial_y^2f}{(1+f) ^2}- \frac{2(\partial_y f )^3}{(1+f)^3}  \right]\partial_x^{m-1}g.
 \end{aligned}
 \end{equation}
 
 \bigskip
  \noindent{\bf Acknowledgements}.
The work was supported by NSF of China (Nos. 12325108,  12131017, 12221001),  and  the Natural Science Foundation of Hubei Province (No. 2019CFA007).

	

\end{document}